\documentclass[a4paper]{paper}

\usepackage[utf8]{inputenc}
\usepackage[T1]{fontenc}
\usepackage{lmodern}
\usepackage{tikz}
  \usetikzlibrary{math}
\usepackage{tkz-euclide}
\usepackage{graphicx}      
\usepackage{pgfplots}
\usepackage{csquotes}
  \pgfplotsset{compat=1.13}
\usepackage{subcaption}
\usepackage{amssymb}
\usepackage{amsmath}
\usepackage{amsthm}
\usepackage{mathtools}
\usepackage{hyperref}
\usepackage{cleveref}
  \crefname{figure}{figure}{figures}
  \Crefname{figure}{Figure}{Figures}
  \crefname{subsection}{subsection}{subsections}
  \Crefname{subsection}{Subsection}{Subsections}
\usepackage{acro}
\usepackage[giveninits=true, style=numeric-comp]{biblatex} 
\usepackage{microtype}
\usepackage{todonotes}

\makeatletter
\providecommand{\maketitle}{}
\renewcommand{\maketitle}{%
	\par
	\begingroup
	\renewcommand{\thefootnote}{\fnsymbol{footnote}}
	\renewcommand{\@makefnmark}{\hbox to \z@{$^{\@thefnmark}$\hss}}
	\long\def\@makefntext##1{%
		\parindent 1em\noindent
		\hbox to 1.8em{\hss $\m@th ^{\@thefnmark}$}##1
	}
	\thispagestyle{empty}
	\@maketitle
	\@thanks
	\endgroup
	\let\maketitle\relax
	\let\thanks\relax
}

\providecommand{\@maketitle}{}
\renewcommand{\@maketitle}{%
	\vbox{%
		\hsize\textwidth
		\linewidth\hsize
		\vskip 0.1in
		\centering
		{\LARGE\bf \@title\par}
		\def\And{%
		\end{tabular}\hfil\linebreak[0]\hfil%
		\begin{tabular}[t]{c}\bf\rule{\z@}{24\p@}\ignorespaces%
		}
		\def\AND{%
		\end{tabular}\hfil\linebreak[4]\hfil%
		\begin{tabular}[t]{c}\bf\rule{\z@}{24\p@}\ignorespaces%
		}
		\begin{tabular}[t]{c}\bf\rule{\z@}{24\p@}\@author\end{tabular}%
		\vskip 0.3in \@minus 0.1in
	}
}
\makeatother

\newcommand{\PrimS}{\ensuremath{\mathcal X}}

\newcommand{\Real}{\ensuremath{\mathbb{R}}}  
\newcommand{\RealExt}{\ensuremath{\overline{\Real}}} 
\newcommand{\defeq}{\ensuremath{\mathrel{\mathop:} =}}        
\newcommand{\Prox}[2]{\ensuremath{\mathrm{Prox}^{#1}_{#2}}} 
\newcommand{\p}[1]{\ensuremath{\mathord{\left(#1\right)}}} 
\newcommand{\inpr}[2]{\ensuremath{\mathord{\left\langle #1, #2 \right\rangle}}} 
\newcommand{\norm}[1]{\ensuremath{\mathord{\left\| #1 \right \|}}} 
 
\newcommand{\weakto}{\ensuremath{\rightharpoonup}} 
\newcommand{\set}[1]{\ensuremath{\mathord{\left\{ #1 \right\}}}} 
\newcommand{\setcond}[2]{\ensuremath{\set{#1\, \middle| \, #2}}} 
\DeclareMathOperator{\dom}{dom} 
\DeclareMathOperator*{\argmin}{arg\,min} 
\newcommand{\tols}[1]{\ensuremath{\alpha_{#1}}}
\newcommand{\tolna}[1]{\ensuremath{\alpha^{(1)}_{#1}}}
\newcommand{\tolno}[1]{\ensuremath{\alpha^{(2)}_{#1}}}
\newcommand{\OptParamSp}{\ensuremath{\mathcal Y}} 
\DeclareMathOperator{\OptSolver}{\mathcal{R}}
\newcommand{\RegFunc}{S}
\newcommand{\normalize}{\mathcal{N}}

\theoremstyle{plain}
  \newtheorem{theorem}{Theorem}[section]
  
  \newtheorem{corollary}{Corollary}[section]
\theoremstyle{definition}
  
  \newtheorem{remark}{Remark}[section]

\DeclareAcronym{CT}{
  short = CT,
  long = computed tomography}  
\DeclareAcronym{WASP}{
  short = WASP,
  long = {Wallenberg AI, Autonomous Systems and Software Program}}
\DeclareAcronym{ISTA}{
  short = ISTA,
  long = Iterative Shrinkage-Thresholding Algorithm}  
\DeclareAcronym{FISTA}{
  short = FISTA,
  long = fast ISTA} 
\DeclareAcronym{LISTA}{
  short = LISTA,
  long = learned ISTA} 
  
\title{Accelerated Forward-Backward Optimization using Deep Learning}
\author{
  Sebastian Banert\thanks{Equal contribution.}\;\thanks{Work done at KTH.} \vspace{0.5mm} \\
  Department of Automatic Control \\
  Lund University \\
  \href{mailto:sebastian.banert@control.lth.se}{\texttt{sebastian.banert@control.lth.se}}
  \And
  Jevgenija Rudzusika\footnotemark[1] \vspace{0.5mm} \\
  Department of Mathematics \\ 
  KTH Royal Institute of Technology\\
  \href{mailto:jevaks@kth.se}{\texttt{jevaks@kth.se}}
  \And 
  Ozan \"Oktem \vspace{0.5mm} \\
  Department of Mathematics \\ 
  KTH Royal Institute of Technology\\
  \href{mailto:ozan@kth.se}{\texttt{ozan@kth.se}}
  \And
  Jonas Adler\thanks{Work partially done at KTH and Elekta.} \vspace{0.5mm} \\
  DeepMind \\
  \href{mailto:jonasadler@google.com}{\texttt{jonasadler@google.com}} \\
}

\date{}

\begin{document}
\maketitle
  
\begin{abstract}
We propose several deep-learning accelerated optimization solvers with convergence guarantees. We use ideas from the analysis
of accelerated forward-backward schemes like FISTA, but instead of the classical approach of proving convergence for a choice of parameters, such as a step-size, we show convergence whenever the update is chosen in a specific set.
Rather than picking a point in this set using some predefined method, we train a deep neural network to pick the best update. Finally, we show that the method is applicable to several cases of smooth and non-smooth optimization and show superior results to established accelerated solvers.
\end{abstract}

\section{Introduction}
The tremendous growth in computing capacity of hardware and advances in numerical methods has allowed for very large-scale and detailed simulations of complex phenomena.
Much focus is currently on embedding such simulation techniques into procedures for design, control, data assimilation, and inverse modeling (solving inverse problems/parameter estimation).
These tasks are often formalized as continuous variable optimization problems (i.e.\@ nonlinear programs). 

An essential characteristic of the resulting optimization problems is their very large problem size, both in the number of state variables and decision variables. This is dictated by the approximation and discretization of the underlying continuous problem and the number of state variables and state equations, which in many scientific and industrial problems easily reaches the millions.
Another further complicating factor is that the objective involved in the optimization may not necessarily be differentiable, like when sparsity is used in inverse modeling.

For the above stated reasons, \emph{solving large-scale non-smooth optimization problems is one of the key challenges in scientific computing}. 
This has catalyzed a surge in research on methods that utilize specific structure in the problem and take advantage of modern multicore or distributed computing architectures. 
An example is decomposition algorithms that allow considerable computational speedup in parallel computing environments while addressing memory requirements. 
Another is operator splitting techniques, which are designed to solve optimization problems where the objective is a sum of two convex functions, one differentiable with a Lipschitz continuous gradient and the other possibly non-smooth but its proximal points are accessible.
These are now very popular for addressing large-scale non-smooth optimization problems that arise in processing massive datasets.
The latest line of development seeks to further improve upon computational feasibility by combining the aforementioned techniques from convex analysis with contemporary techniques from deep learning.  

\subsection{Mathematical setting}
Let $\PrimS$ denote a Hilbert space with inner product $\inpr{\cdot}{\cdot}$ and corresponding norm $\norm{\cdot}$.
Next, consider minimizing a differentiable convex function $f \colon \PrimS \to \Real$  whose gradient $\nabla f \colon \PrimS \to \PrimS$ is Lipschitz continuous with constant $\beta^{-1} > 0$.
This can be solved by gradient descent \cite{Cauchy:1847} or accelerated methods based on an overshooting idea formulated by Nesterov \cite{Nesterov:1983}. 

The objective in the aforementioned minimization is smooth, but as already mentioned, many applications involve minimizing a non-smooth objective function.
To handle the latter, there has been significant effort to develop theory and algorithms suitable for non-smooth optimization. 
These often utilize further structure in the objective. In particular, \emph{splitting} techniques have been developed to solve optimization problems where the (non-smooth) objective function is a sum of two convex functions, one smooth and the other non-smooth.
Stated mathematically, this refers to optimization problems of the form 
\begin{equation}\label{eq:smooth_non-smooth_problem}
  \min_{x\in \PrimS} f\p{x} + g\p{x}
\end{equation}
where $f \colon \PrimS \to \Real$ is a differentiable convex function and $g \colon \PrimS \to \RealExt$ is a proper, convex and lower semi-continuous, but possibly non-smooth function.
Algorithms to solve \cref{eq:smooth_non-smooth_problem} include the forward-backward splitting method \cite{Goldstein:1964} and its accelerated variants like \ac{FISTA} \cite{BeckTeboulle:2009}. 
Much attention has also been directed towards approaches that use learning techniques to improve the speed, like in \ac{LISTA} \cite{GregorLecun:2010}, see the survey \cite{chen2021learning} and references therein.
However, the latter methods lack convergence guarantees. 
Another approach is to learn the parameters of a general optimization solver \cite{BanertRinghAdlerKarlssonOktem:2018}, but in this case there is a relatively small number of free parameters, which limits the potential of the scheme to be adapted to a particular problem class. 
To address this and increase the number of degrees of freedom in the learning, we will in this paper consider a \emph{deviation-based} approach. 
Instead of learning parameters of a solver, we learn an entire updating function and then use that in an optimization solver.
To ensure convergence, we constrain the updating function to sufficiently small deviations from known standard algorithms.

\subsection{Related work}
The overall principle that underlies application of machine learning for solving optimization problems is to use principles from statistical decision theory to select the \enquote{best} solver by training against data. 
Solvers are parametrized and \enquote{best} means selecting the solver parameters that lead to the smallest average error when applied to a collection of similar optimization problems.

Several authors have investigated the idea of deriving optimization schemes from statistical learning. One of the most notable schemes is \ac{LISTA} introduced in \cite{GregorLecun:2010}.
This scheme was inspired by the \ac{ISTA}, which is a popular algorithm for solving the optimization in sparse coding. 
Each iteration of \ac{ISTA} is a combination of matrix operations followed by a thresholding function. 
Thus, when stopped after fixed number of iterations, the algorithm resembles a forward pass through a neural network. 
The idea in \ac{LISTA} is to replace the handcrafted matrices in \ac{ISTA} with  learned ones that are trained in a supervised manner, provided that the ground truth solution to the optimization problem is available. 
In contrast, we follow \cite{BanertRinghAdlerKarlssonOktem:2018} and train the parameters of our optimization scheme using an unsupervised loss that is equal to the value of the objective function attained after a fixed number of iterations.

The above idea of learning a truncated optimization scheme in a unsupervised manner is also studied in machine learning  \cite{DBLP:journals/corr/AndrychowiczDGH16, li2016learning,wichrowska2017learned,bello2017neural,Maheswaranathan:2020aa}. 
The main use case here is to solve the non-convex optimization problem that arises when training neural networks.
Here the objective is to replace conventional optimizers with learned optimizers that are faster/better.
This is achieved by parametrizing updates at each iteration with recurrent neural networks and training the associated hyper-parameters either by variants of stochastic gradient descent \cite{wichrowska2017learned,DBLP:journals/corr/AndrychowiczDGH16}, reinforcement learning \cite{li2016learning,bello2017neural} or by using evolutionary strategies \cite{metz2020tasks}. 
The challenge is to ensure the learned optimizer generalizes from one optimization task to another. 
A difficulty here is to identify a class of optimization problems that share similar structure.
We instead consider (large) classes of optimization problems that \textit{naturally} share similarities, like those in image reconstruction for \ac{CT} and regularization of other inverse problems. 
This makes it easier to specify the generalization requirement.

Convergence of the above mentioned learned optimization schemes is yet to be mathematically proven.
Instead, their usage relies on (very strong) heuristics, so these data driven approaches can not be seen as optimization solvers in the classical sense. As we consider convex problems, our approach is not applicable for training neural networks, but we can on the other hand provide convergence guarantees.

Some authors have investigated optimal parameter choices in \ac{FISTA} when iterations are truncated, see e.g.\@ \cite{liang2018improving} and notably \cite{DBLP:journals/corr/abs-1808-10038} that proves linear convergence of a related scheme under some further assumptions (sparsity, finite dimension, optimal training). Another approach is \cite{BanertRinghAdlerKarlssonOktem:2018}, which proves convergence of a learned primal-dual scheme under weaker assumptions. The analysis is however only performed in the parametric case, e.g.\@ with an explicit form of the updating operation.

A key element in the above approaches to optimization lies in using neural networks obtained by unrolling an existing optimization scheme.
This provides the opportunity for domain adaptation and such network architectures have also been successfully used to solve challenging inverse problems, but this typically involves training them in a supervised manner \cite{Adler:2017aa, Adler:2018aa, PockMRI, DeepADMMNET}. 
Note that even though these networks have an architecture obtained by unrolling an optimization solver, they are not themselves optimization solvers, see \cite[Section~5]{Arridge:2019aa} for further details.  
Another example from scientific computing is \cite{hsieh2018learning}, which deals with solving partial differential equations. 
Hence, unrolling should primarily be seen as a way to assemble an appropriately domain adapted deep neural network architecture \cite[Section~4.9.1]{Arridge:2019aa}.

An attempt to analyze and interpret learned optimizers by considering linear approximations of optimizers close to the stationary point is presented in \cite{Maheswaranathan:2020aa}. 

\subsection{Overview of the paper}
\Cref{sec:background} recalls the necessary notions from convex optimization. 
This is followed by \cref{sec:smooth}, which establishes a worst-case convergence result for a non-parametric gradient descent scheme for smooth optimization problems.
Next, \cref{sec:non-smooth} establishes a slightly weaker convergence theorem for a non-parametric forward-backward scheme applicable for non-smooth problems. \Cref{sec:methods} discusses the details of learning a non-parametric optimization scheme. \Cref{sec:applications} concludes with numerical experiments from an imaging application.

\section{Background from optimization theory}\label{sec:background}
The aim here is to introduce the notation and basic notions from optimization theory that are used throughout the paper. The main reference is  \cite{BauschkeCombettes:2017}, but the reader may also consult \cite{Nesterov:2004, Rockafellar:1970} for related results.

A function $f\colon \PrimS \to \RealExt = \Real \cup \set{\pm\infty}$ is
\emph{proper} if its domain is nonempty and it does not take the value $-\infty$, i.e.\@ if
\[
\dom f \defeq \setcond{x\in \PrimS}{f\p{x} < +\infty} \neq \varnothing
\quad\text{and}\quad
f\p{x} > -\infty \text{ for all $x\in \PrimS$.}
\]
It is \emph{convex} if 
$
f\p{\p{1 - \lambda} x + \lambda y} \leq \p{1 - \lambda} f\p{x} + \lambda
f\p{y}
$
holds for all $x, y\in \PrimS$ and $0 \leq \lambda \leq 1$ and \emph{strictly convex} if this inequality is strict for $x \neq y$ and $0 < \lambda < 1$. A strictly convex function has at most one minimizer \cite[Corollary~11.9]{BauschkeCombettes:2017}.
Next, the \emph{sub-differential} of a convex function $f\colon \PrimS \to \RealExt$ is a (set-valued) mapping $\partial f \colon \PrimS \to 2^{\PrimS}$ defined as 
\[
\partial f\p{x} \defeq 
\begin{cases}
\setcond{y\in \PrimS}{\forall z\in \PrimS: f\p{z}
\geq f\p{x} + \inpr{y}{z - x}}
& \text{if $f\p{x} \in \Real$}
\\
\emptyset & \text{otherwise.}
\end{cases}
\]
If $f$ is Fr\'echet differentiable at $x$, then $\partial f\p{x}$ consists of a single element, namely the gradient of $f$ at $x$ \cite[Proposition 17.31]{BauschkeCombettes:2017}, i.e.\@ $\partial f\p{x} = \set{\nabla f\p{x}}$.
Finally, the \emph{proximal point} of a proper, convex, and lower semi-continuous function $f\colon \PrimS \to \RealExt$ at $x\in\PrimS$ with \emph{step-size} $\gamma > 0$ is defined as
\[
\Prox{\gamma}{f} \p{x} \defeq  \argmin \setcond{f\p{z} + \frac{1}{2\gamma} \norm{z - x}^2}{z\in \PrimS}.
\]
The proximal point can be used to characterize points in the sub-differential \cite[Proposition 16.44]{BauschkeCombettes:2017}:
\[ p = \Prox{\gamma}{f}\p{x}
   \iff 
   \frac{1}{\gamma}\p{x - p} \in \partial f\p{p}.
\]   
The operator $\Prox{\gamma}{f}$ is then a \emph{firmly non-expansive} single-valued mapping from $\PrimS$ to $\PrimS$ \cite[Proposition 12.28]{BauschkeCombettes:2017},
\begin{multline*}
\norm{\Prox{\gamma}{f}\p{x_1} - \Prox{\gamma}{f}\p{x_2}}^2 \\
\leq \inpr{\Prox{\gamma}{f}\p{x_1} - \Prox{\gamma}{f}\p{x_2}}{x_1 - x_2}
\text{ for all $x_1, x_2 \in \PrimS$.} 
\end{multline*}
Furthermore, $p = \Prox{\gamma}{f}\p{x}$ if and only if the following variational inequality holds \cite[Proposition 12.26]{BauschkeCombettes:2017}:
\begin{equation}
f\p{z} \geq f\p{p} + \frac{1}{\gamma} \inpr{x - p}{z - p}
\quad\text{for all $z\in\PrimS$.}
\end{equation}

If $\nabla f$ is Lipschitz continuous with constant $\beta^{-1}$, then the following inequality, which is sometimes called \emph{descent lemma} (see \cite[Theorem 18.15]{BauschkeCombettes:2017}), holds:
\begin{equation}\label{eq:descent_lemma}
  f\p{x + d} \leq f\p{x} + \inpr{\nabla f\p{x}}{d} + \frac{1}{2\beta} \norm{d}^2
  \quad\text{$x, d \in \PrimS$.}
\end{equation}
Likewise, if $f$ is in addition assumed to be convex \cite[Theorem 2.1.5]{Nesterov:2004}:
\begin{equation}\label{eq:Lipschitz_estimation}
  f\p{x} + \inpr{\nabla f\p{x}}{y - x} + \frac{\beta}{2} \norm{\nabla f\p{x} - \nabla f\p{y}}^2 \leq f\p{y}
  \quad\text{for all $x, y \in \PrimS$,}
\end{equation}

\section{Deviation-based optimization schemes}
We call an optimization scheme \emph{parameter-based} if convergence criteria are expressed as constraints on its parameters. 
An example of a parameter-based optimization scheme is the classical gradient descent method for minimizing a convex differentiable function $f: \PrimS \to \Real$ where $\nabla f\colon \PrimS \to \PrimS$ is $\beta^{-1}$-Lipschitz. To see this, note that such a method relies on an updating scheme of the following form:
\begin{equation}\label{eq:steepest_descent}
  x_{n+1} \defeq x_n - \beta \p{1 + t_n} \nabla f\p{x_n}.
\end{equation}
For convergence to a minimizer for $f$, $t_n$ is a parameter that needs to satisfy $-1 < t_n < 1$ and $\sum_{n = 0}^\infty \p{1  + t_n} \p{1 - t_n} = +\infty$ \cite[Proposition~4.39, Corollary~5.16, and Corollary~18.17]{BauschkeCombettes:2017}. 
Thus, a convergence criterion for the above scheme is expressible in terms of a specific choice of parameters, so we are dealing with a parameter-based optimization method. 

In fact, almost all methods in classical optimization theory are written as schemes of this form. 
Furthermore, it has been considered a virtue to have as few parameters as possible (preferably none).
On the other hand, recent advancements in machine learning allow us to deal with extremely large numbers of parameters while avoiding over-fitting effects and achieve significant improvements over established methods. While optimization schemes can be formulated using neural networks, and their convergence can be studied in terms of their parameters, this is typically very cumbersome and restricts neural network architectures to, e.g., one layer networks \cite{BanertRinghAdlerKarlssonOktem:2018}.

To allow for a wider range of neural network architectures, we need to move away from convergence guarantees that involve the parameters of the scheme.
The idea is to give the network as much freedom as possible while still retaining the stability properties of the underlying optimization method. 
Such schemes that do not impose an a-priori restriction on the number of parameters will be referred to as \emph{deviation-based}.

To introduce a deviation-based scheme for smooth optimization as in the setting above, we modify the steepest descent method~\cref{eq:steepest_descent} by introducing a sequence $\p{\Delta x_n}_{n \geq 0}$ of deviations:
\begin{equation}\label{eq:gradient-deviation-based}
    x_{n+1} \defeq x_n - \beta\p{\nabla f\p{x_n} + \Delta x_n}
    \quad\text{for some deviation/offset $\Delta x_n \in \PrimS$.}
\end{equation}
The classical convergence proof for the gradient descent implies convergence if $\Delta x_n = t \nabla f\p{x_n}$ for $-1 < t < 1$. 
This is, however, a one-dimensional set and hardly interesting to do learning on, and in particular not deep learning.
In what follows, we consider schemes of the above form and prove convergence for sets of updates $\Delta x_n \in \PrimS$ that have full dimension, thus allowing for a much greater degree of freedom in selecting the updates. 
The learning itself will be discussed in~\cref{sec:methods}.

\subsection{Deviation-based optimization schemes for smooth optimization}\label{sec:smooth}
In previous literature, the sequence $\p{\Delta x_n}_{n \geq 0}$ of deviations in~\cref{eq:gradient-deviation-based} was used to model errors in the evaluation of the gradient, see for example \cite{Briceno-AriasCombettes:2011, CombettesPesquet:2012, Vu:2013}. In our setting, however, we are interested in consciously choosing $\Delta x_n \in \PrimS$ of~\cref{eq:steepest_descent} which will preserve or, by a clever choice, might improve its convergence properties.

As we show below the condition $\norm{\Delta x_n} \leq \norm{\nabla f\p{x_n}}$ guarantees that the function value $f\p{x_n}$ is non-increasing, and it is the weakest condition which guarantees this property.
In particular, we will show that the function values $f\p{x_n}$ converge to the minimal value of the function $f$.
Special cases of the following theorem are discussed in~\cref{rem:smooth_special_cases} below.

\begin{theorem}\label{thm:smooth_main_convergence}
  Consider minimizing $f\colon \PrimS \to \Real$ that is convex and differentiable with a gradient that is $\beta^{-1}$-Lipschitz continuous for some $\beta > 0$.
  Next, consider the iterative scheme
  \begin{equation}\label{eq:smooth_iteration}
    x_{n+1} \defeq x_n - \beta\p{\nabla f\p{x_n} + \Delta x_n}
    \quad\text{for all $n\geq 0$,}
  \end{equation}
  where $x_0 \in \PrimS$ is some fixed initial point and $\p{\Delta x_n}_{n\geq 0}$ is any sequence in $\PrimS$ such that
  \begin{equation}\label{eq:smooth_limsup}
    \norm{\Delta x_n} \leq \tols{n} \norm{\nabla f\p{x_n}}
  \end{equation}
  for some sequence $\p{\tols{n}}_{n \geq 0}$ of positive numbers with $\tols{n} \leq 1$ for all $n \geq 0$.
  \begin{enumerate}
    \item \label{item:thm:smooth_main_convergence:decreasing} 
      Then $\p{f\p{x_n}}_{n\geq 0}$ is a non-increasing sequence of objective function values.
    \item \label{item:thm:smooth_main_convergence:rate}
      The following inequality holds:
      \begin{equation}\label{eq:smooth_variable_rate}
        f\p{x_n} - f\p{\bar x}\leq \frac{1}{2\beta} \p{\prod_{k = 0}^{n - 1} \p{1 - \frac{1 - \tols{k}^2}{k + 2}}} \norm{x_0 - \bar x}^2
        \text{ for all $n\geq 0$,}
      \end{equation}
      where $\bar x$ denotes a global minimizer of $f$.
    \item \label{item:thm:smooth_main_convergence:objective_convergence}
      If we further assume that
      \begin{equation}\label{eq:smooth_nonsummability_condition}
        \sum_{n = 0}^\infty \frac{1 - \tols{n}}{n + 2} = + \infty.
      \end{equation}
      Then $f\p{x_n} \to f\p{\bar x}$ as $n\to +\infty$ with $\bar x$ denoting a global minimizer of $f$.
    \end{enumerate}
\end{theorem}

\begin{remark}\label{rem:smooth_special_cases}
Choosing $\tols{n} = \tols{}$ for all $n$ in~\cref{thm:smooth_main_convergence} simplifies \cref{eq:smooth_variable_rate} to
\begin{equation}\label{eq:smooth_fixed_rate}
\begin{split}
    f\p{x_n} - f\p{\bar x}
    &\leq \frac{1}{2\beta} \p{\prod_{k = 0}^{n - 1} \p{\frac{k + 1 + \tols{}}{k + 2}}} \norm{x_0 - \bar x}^2
    \\
    &= \frac{\Gamma\p{n + \tols{} + 1}}{2 \beta \Gamma \p{n + 2} \Gamma \p{1 + \tols{}}} \norm{x_0 - \bar x}^2.
\end{split}
\end{equation}
Hence, using Stirling's formula,
\[
    \frac{f\p{x_n} - f\p{\bar x}}{\norm{x_0 - \bar x}^2} = O\p{n^{- \p{1 - \tols{}}}} \qquad \text{as } n \to \infty.
\]
When $\tols{} = 0$, i.e.\@ if $\Delta x_n = 0$ for $n = 0, 1, \dots$, we get
\[
    f\p{x_n} - f\p{\bar x} = \frac{1}{2 \beta \p{n + 1}} \norm{x_0 - \bar x}^2.
\]
Hence, we have (asymptotically) recovered the optimal rate from~\cite[Theorem 1]{DroriTeboulle:2014} up to a constant of $2$.

Note also that all the calculations work for a smaller $0 < \beta' <\beta$ since  $1/\beta'$ will also be a Lipschitz constant of $\nabla f$. In addition to the convergence of the function values, we are also able to show the convergence of the iteration sequence with stronger assumptions on the objective function $f$:
\end{remark}
\begin{corollary}\label{cor:smooth}
If, in addition to the assumptions in~\cref{item:thm:smooth_main_convergence:objective_convergence} of~\cref{thm:smooth_main_convergence}, the objective function $f$ is coercive, i.e., $f\p{x} \to +\infty$ as $\norm{x} \to +\infty$, and strictly convex, then the sequence $\p{x_n}_{n\geq 0}$ converges weakly to the unique minimizer of $f$.
\end{corollary}
\begin{proof}
The sequence $\p{x_n}_{n \geq 0}$ is bounded: by~\cref{item:thm:smooth_main_convergence:decreasing} of~\cref{thm:smooth_main_convergence}, the sequence $\p{f\p{x_n}}_{n\geq 0}$ is bounded from above by $f\p{x_0} \in \Real$. The asserted boundedness then follows from the coercivity of $f$.

If $x\in \PrimS$ is a weak sequential cluster point, i.e., there exists a subsequence $\p{x_{n_k}}_{k \geq 0}$ of $\p{x_n}_{n \geq 0}$ with $x_{n_k} \weakto x$, then \cite[Theorem 9.1]{BauschkeCombettes:2017} and~\cref{item:thm:smooth_main_convergence:objective_convergence} of~\cref{thm:smooth_main_convergence} give $f\p{x} \leq \lim_{n \to \infty} f\p{x_{n_k}} = \inf_{x \in \PrimS} f\p{x}$, i.e., $x$ is the unique minimizer of $f$. The statement of the corollary then follows from \cite[Lemma~2.46]{BauschkeCombettes:2017}.
\end{proof}

\begin{proof}[Proof of~\cref{thm:smooth_main_convergence}]
  For all $n \geq 0$ we have, by \cref{eq:Lipschitz_estimation},
  \begin{align}\label{eq:smooth:consecutive}
    \hspace{-0.3cm}
    f\p{x_{n+1}} - f\p{x_n}
    &\leq 
    \inpr{\nabla f\p{x_{n+1}}}{x_{n+1} - x_n} 
       - \frac{\beta}{2} \norm{\nabla f\p{x_{n+1}} - \nabla f\p{x_n}}^2 
    \nonumber \\
    &= -\beta \inpr{\nabla f\p{x_{n+1}}}{\Delta x_n} 
       - \frac{\beta}{2} \norm{\nabla f\p{x_{n+1}}}^2 
       - \frac{\beta}{2} \norm{\nabla f\p{x_n}}^2 
    \nonumber \\
    &=  - \frac{\beta}{2} \norm{\nabla f\p{x_{n+1}} 
       + \Delta x_n}^2 + \frac{\beta}{2} \p{\norm{\Delta x_n}^2 
       - \norm{\nabla f\p{x_n}}^2}. 
  \end{align}
  The proof of \cref{item:thm:smooth_main_convergence:decreasing} now follows from the observation that 
  \[
    \norm{\Delta x_n}^2 - \norm{\nabla f\p{x_n}}^2 \leq - \p{1 - \tols{n}^2}
    \norm{\nabla f\p{x_n}}^2.
  \]
  
  We will next prove the formula in \cref{eq:smooth_variable_rate} by proving the following stronger relation:
  \begin{multline}\label{eq:smooth_strong_rate}
    f\p{x_n} - f\p{\bar x} 
    \leq \frac{1}{2\beta} \p{\prod_{k = 0}^{n - 1} 
      \p{1 - \frac{1 - \tols{k}^2}{k + 2}}} \norm{x_0 - \bar x}^2 
    \\
    - \frac{1}{2 \beta \p{n + 1}} \norm{x_n - \bar x - \beta \nabla f\p{x_n}}^2
  \end{multline}
  where $\bar x \in \PrimS$ denotes a global minimizer of $f$, i.e.\@ $\nabla f\p{\bar x} = 0$. 
  To prove the above, we proceed by induction and begin by considering the case $n = 0$. 
  By \cref{eq:Lipschitz_estimation}, we have
  \begin{align*}
    f\p{x_0} - f\p{\bar x}
    &\leq \inpr{\nabla f\p{x_0}}{x_0 - \bar x} - \frac{\beta}{2} \norm{\nabla f\p{x_0}}^2 \\
    &= \frac{1}{2 \beta} \norm{x_0 - x}^2 - \frac{1}{2 \beta} \norm{x_0 - \bar x - \beta \nabla f\p{x_0}}^2.
  \end{align*}
  
  This proves the claim for $n=0$, i.e.\@ the induction basis holds.
  Next, assume \cref{eq:smooth_strong_rate} holds for some $n \geq 0$ and we seek to prove that it holds for $n+1$. Then
  \begin{multline*}
    f\p{x_{n+1}} - f\p{\bar x} = \frac{1 - \tols{n}^2}{n + 2} \p{f\p{x_{n+1}} -
      f\p{\bar x}} \\
    + \p{1 - \frac{1 - \tols{n}^2}{n + 2}} \p{f\p{x_{n+1}} - f\p{x_n}} + \p{1 - \frac{1 - \tols{n}^2}{n + 2}} \p{f\p{x_n} - f\p{\bar x}}
  \end{multline*}
  Here, we estimate the first two terms with \cref{eq:Lipschitz_estimation} and
  the last one with the induction hypothesis \cref{eq:smooth_strong_rate}. This
  results in
  \begin{multline*}
    f\p{x_{n+1}} - f\p{\bar x} \leq \frac{1 - \tols{n}^2}{n + 2} \p{\inpr{\nabla
        f\p{x_{n+1}}}{x_{n+1} - \bar x} - \frac{\beta}{2} \norm{\nabla f\p{x_{n+1}}}^2} 
    \\ \shoveleft{\qquad
    + \p{1 - \frac{1 - \tols{n}^2}{n + 2}} \p{\inpr{\nabla f\p{x_{n+1}}}{x_{n+1}
      - x_n} - \frac{\beta}{2} \norm{\nabla f\p{x_{n+1}} - \nabla f\p{x_n}}^2} 
    }
    \\ \shoveleft{\qquad
    + \frac{1}{2\beta} \p{\prod_{k = 0}^n 
        \p{1 - \frac{1 - \tols{k}^2}{k + 2}}} \norm{x_0 - \bar x}^2 
  }
  \\
  - \p{1 - \frac{1 - \tols{n}^2}{n + 2}} 
        \frac{1}{2 \beta \p{n + 1}} 
        \norm{x_n - \bar x - \beta \nabla f\p{x_n}}^2.
  \end{multline*}
  Using \cref{eq:smooth_iteration} and reordering the terms gives
  \begin{multline}\label{eq:smooth_proof_1}
    f\p{x_{n+1}} - f\p{\bar x} \leq  - \beta \inpr{\nabla f\p{x_{n+1}}}{\Delta x_n} - \frac{\beta}{2} \norm{\nabla f\p{x_{n+1}}}^2 
    \\ \shoveleft{\qquad
    + \frac{1 - \tols{n}^2}{n + 2} \inpr{\nabla f\p{x_{n+1}}}{x_n - \bar x - \beta \nabla f\p{x_n}}
    } 
    \\ \shoveleft{\qquad
    - \p{1 - \frac{1 - \tols{n}^2}{n + 2}} \p{\frac{\beta}{2} \norm{\nabla f\p{x_n}}^2 + \frac{1}{2 \beta \p{n + 1}} \norm{x_n - \bar x - \beta \nabla f\p{x_n}}^2} 
    }
    \\
    + \frac{1}{2\beta} \p{\prod_{k = 0}^n \p{1 - \frac{1 - \tols{k}^2}{k + 2}}}
    \norm{x_0 - \bar x}^2.
  \end{multline}
  On the other hand, using again \cref{eq:smooth_iteration},
  \begin{multline*}
    \frac{1}{2\beta \p{n + 2}} \norm{x_{n+1} - \bar x - \beta \nabla f\p{x_{n+1}}}^2 \\
    = \frac{1}{2\beta \p{n + 2}} \norm{x_n - \bar x - \beta \nabla
      f\p{x_n}}^2 + \frac{\beta}{2 \p{n + 2}} \norm{\Delta x_n + \nabla
      f\p{x_{n+1}}}^2 \\
    - \frac{1}{n + 2} \inpr{x_n - \bar x - \beta \nabla f\p{x_n}}{\Delta x_n + \nabla f\p{x_{n+1}}}.
  \end{multline*}
  Adding this to \cref{eq:smooth_proof_1}, followed by a simple calculation, yields
  \begin{multline*}
    f\p{x_{n+1}} - f\p{\bar x} + \frac{1}{2\beta \p{n + 2}} \norm{x_{n+1} - \bar
      x - \beta \nabla f\p{x_{n+1}}}^2 
    \\ \shoveleft{\quad
    \leq \frac{1}{2\beta} \p{\prod_{k = 0}^n \p{1 - \frac{1 - \tols{k}^2}{k + 2}}} \norm{x_0 - \bar x}^2 - \frac{\beta \p{n + 1}}{2 \p{n + 2}} \p{1 - \tols{n}^2} \norm{\nabla f\p{x_{n+1}}}^2
    }
    \\ \shoveleft{\qquad
    - \frac{\tols{n}^2}{2 \beta \p{n + 1} \p{n + 2}} \norm{x_n - \bar x - \beta \nabla f\p{x_n} + \frac{\beta \p{n + 1}}{\tols{n}^2} \Delta x_n + \beta \p{n + 1} \nabla f\p{x_{n+1}}}^2 
    }
    \\
    - \frac{\beta}{2} \p{1 - \frac{1 - \tols{n}^2}{n + 2}} \p{\norm{\nabla f\p{x_n}}^2 - \frac{1}{\tols{n}^2} \norm{\Delta x_n}^2}.
  \end{multline*}
  By \cref{eq:smooth_limsup} and taking into account the non-negativity of
  norms, we get
  \begin{multline*}
    f\p{x_{n+1}} - f\p{\bar x} + \frac{1}{2\beta \p{n + 2}} \norm{x_{n+1} - \bar
      x - \beta \nabla f\p{x_{n+1}}}^2 \\
    \leq \frac{1}{2\beta} \p{\prod_{k = 0}^n \p{1 - \frac{1 - \tols{k}^2}{k + 2}}} \norm{x_0 - \bar x}^2,
  \end{multline*}
  which finishes the induction. Let us now prove the convergence $f\p{x_n} \to
  f\p{\bar x}$ as $n \to +\infty$. Taking the logarithm in
  \cref{eq:smooth_variable_rate} gives
  \[
    \log \p{f\p{x_n} - f\p{\bar x}} \leq \sum_{k = 0}^{n - 1} \log \p{1 - \frac{1 - \tols{k}^2}{k + 2}} + \log \norm{x_0 - \bar x}^2 - \log \p{2\beta}.
  \]
  Since $\log$ is a concave function, $\log\p{1 - x} \leq -x$ for all $x \in
  \Real$. Hence,
  \begin{align*}
    \log \p{f\p{x_n} - f\p{\bar x}}
    &\leq - \sum_{k = 0}^{n - 1} \frac{1 - \tols{k}^2}{k + 2} + \log \norm{x_0 - \bar x}^2 - \log \p{2\beta} \\
    &\leq - \sum_{k = 0}^{n - 1} \frac{1 - \tols{k}}{k + 2} + \log \norm{x_0 - \bar x}^2 - \log \p{2\beta}.
  \end{align*}
  Because of the assumption in \cref{eq:smooth_nonsummability_condition}, the right-hand side converges to $-
  \infty$, so $f\p{x_n} \to f\p{\bar x}$ as $n \to \infty$.
\end{proof}

\subsection{Deviation-based schemes for non-smooth optimization}\label{sec:non-smooth}
The theory in \cref{sec:smooth} only applies to optimization problems with a smooth objective.
Theory for optimization with a non-smooth objective needs to be based on sub-differential calculus.
Our focus is on minimizing a non-smooth objective with an additional structure as in \cref{eq:smooth_non-smooth_problem}.

The proposed deviation-based iterative scheme is inspired by the classical forward-backward (or proximal-gradient/\ac{ISTA}) scheme \cite{Goldstein:1964}, but with deviations in the spirit of the gradient method outlined in~\cref{sec:smooth}. 
The next theorem states the scheme and analyses its convergence properties for solving the non-smooth minimization problem in \cref{eq:smooth_non-smooth_problem} where $g$ is not necessarily differentiable.
\begin{theorem}\label{thm:non-smooth_main}
  Consider minimizing $f+g$ where $f: \PrimS \to \Real$ is convex and differentiable with a gradient $\nabla f: \PrimS \to \PrimS$ that is Lipschitz continuous with constant $\beta^{-1}$ for some $\beta > 0$, and $g: \PrimS \to \RealExt$ is a proper, convex and lower semi-continuous function that is not necessarily differentiable.
  Next, consider the iterative scheme
  \begin{subequations}\label{eq:non-smooth_iteration}
  \begin{align}
  w_n &\defeq x_n + \Delta x_n^1, \label{eq:non-smooth_iteration_grad} \\
  x_{n+1} &\defeq \Prox{\gamma_n}{g}\p{x_n - \gamma_n \nabla f\p{w_n} + \frac{\gamma_n}{\beta} \Delta x_n^1 + \Delta x_n^2}. \label{eq:non-smooth_iteration_prox}
  \end{align}
  \end{subequations}
  where $x_0 \in \PrimS$ is some fixed initial point, $0 < \gamma_n < 2 \beta$ for all $n\geq 0$, and
  \begin{multline}\label{eq:non-smooth_increment_bounds}
    \frac{1}{2 \beta} \norm{\Delta x_n^1}^2 + \frac{\beta}{2 \gamma_n \p{2 \beta - \gamma_n}} \norm{\Delta x_n^2}^2 \\
    \leq \frac{\tolna{n} \p{2 \beta - \gamma_{n-1}}}{2 \beta \gamma_{n-1}} \norm{x_n - x_{n-1} - \frac{\beta}{2 \beta - \gamma_{n-1}} \Delta x_{n-1}^2}^2 \\
    + \frac{\beta \tolno{n}}{2} \norm{\nabla f\p{w_n} - \nabla f\p{w_{n-1}} - \frac{1}{\beta} \p{x_n - w_{n-1}}}^2
  \end{multline}
  holds for all $n\geq 1$. Then, the following holds:
  \begin{enumerate}
    \item \label{item:thm:non-smooth_Lyapunov}
    For $n \geq 0$, define  
    \[ V_n \defeq f\p{w_n} + g\p{x_{n+1}} + \inpr{\nabla f\p{w_n}}{x_{n+1} - w_n} + \dfrac{1}{2 \beta} \norm{x_{n+1} - w_n}^2.
    \]
    Then $V_n \geq f\p{x_{n+1}} + g\p{x_{n+1}}$ for all $n \geq 0$ and the sequence
      \begin{equation}\label{eq:VnSeq}
        \p{V_n + \frac{2 \beta - \gamma_n}{2 \beta \gamma_n} \norm{x_{n+1} - x_n - \frac{\beta}{2 \beta - \gamma_n} \Delta x_n^2}^2}_{n\geq 0}
      \end{equation}
      is monotonically non-increasing if $\tolna{n} \leq 1$ and $\tolno{n} \leq 1$ for all $n \geq 0$;
    \item \label{item:thm:non-smooth_summable} 
    If $f + g$ is bounded from below and
      \begin{align}
        0 < \liminf_{n\to\infty} \gamma_n \leq \limsup_{n\to\infty} \gamma_n < 2 \beta \label{eq:non-smooth_stepsize_limsup} \\
        0 \leq \limsup_{n\to \infty} \tolna{n} < 1 \qquad \text{and} \qquad 0 \leq \limsup_{n\to\infty} \tolno{n} < 1 \label{eq:non-smooth_radius_limsup}
      \end{align}
      holds, then 
      \begin{alignat*}{2}
        \sum_{n = 0}^\infty \norm{x_n - x_{n+1}}^2 &< +\infty, &\qquad\qquad
        \sum_{n = 0}^\infty \norm{\Delta x_n^1}^2 &< +\infty, \\
        &\text{and} & \sum_{n = 0}^\infty \norm{\Delta x_n^2}^2 &< + \infty. 
      \end{alignat*}

    \item \label{item:thm:non-smooth_cluster_ponts} 
    Under the conditions of \cref{item:thm:non-smooth_summable}, any weak sequential cluster point of $\p{x_n}_{n\geq 0}$ is a minimizer of $f + g$. If there are no such minimizers, then $\norm{x_n} \to +\infty$ as $n\to\infty$.
    Conversely, a weak sequential cluster point exists whenever the sequence $\p{x_n}_{n\geq 0}$ is bounded. 
  \end{enumerate}
\end{theorem}
\begin{corollary}\label{cor:non-smooth}
If, in addition to the assumptions in\cref{item:thm:non-smooth_cluster_ponts} of~\cref{thm:non-smooth_main}, the objective function $f + g$ is coercive, i.e., $f\p{x} + g\p{x} \to +\infty$ as $\norm{x} \to +\infty$, and strictly convex, then the sequence $\p{x_n}_{n\geq 0}$ converges weakly to the unique minimizer of $f + g$.
\end{corollary}
\begin{proof}
The sequence $\p{x_n}_{n \geq 0}$ is bounded: by~\cref{item:thm:non-smooth_Lyapunov} of~\cref{thm:non-smooth_main}, the non-increasing real sequence in~\eqref{eq:non-smooth_lyapunov_monotonicity} guarantees that $\p{V_n}_{n\geq 0}$ is bounded from above and greater or equal to $\p{f\p{x_{n+1}} + g\p{x_{n+1}}}_{n\geq 0}$. The asserted boundedness then follows from the coercivity of $f + g$.

The statement of the corollary then follows from \cite[Lemma~2.46]{BauschkeCombettes:2017}, taking into account~\cref{item:thm:non-smooth_cluster_ponts} of~\cref{thm:non-smooth_main} and  that the strict convexity of $f + g$ ensures the uniqueness of the minimizer of $f + g$.
\end{proof}
\begin{proof}[Proof of~\cref{thm:non-smooth_main}]
  We start by proving the claims in \cref{item:thm:non-smooth_Lyapunov}.
  The first inequality follows from \cref{eq:descent_lemma}, so we focus on proving the monotonic decrease of the sequence in \cref{eq:VnSeq}.
  Note first that, for $n \geq 0$,  \cref{eq:non-smooth_iteration_prox} is equivalent to
  \begin{equation}\label{eq:non-smooth_inclusion}
    \frac{x_n - x_{n+1}}{\gamma_n} - \nabla f\p{w_n} + \frac{1}{\beta} \Delta x_n^1 + \frac{1}{\gamma_n} \Delta x_n^2 \in \partial g\p{x_{n+1}}.
  \end{equation}
  This yields
  \begin{multline}
    g\p{x_n} \geq g\p{x_{n+1}} + \frac{1}{\gamma_n} \norm{x_n - x_{n+1}}^2 - \inpr{\nabla f\p{w_n}}{x_n - x_{n+1}} \\
    + \frac{1}{\beta} \inpr{\Delta x_n^1}{x_n - x_{n+1}} + \frac{1}{\gamma_n} \inpr{\Delta x_n^2}{x_n - x_{n+1}}. \label{eq:non-smooth_g}
  \end{multline}
  By \cref{eq:Lipschitz_estimation} we get the following inequality for all $n \geq 1$:
  \begin{equation}
    f\p{w_{n-1}} \geq f\p{w_n} + \inpr{\nabla f\p{w_n}}{w_{n-1} - w_n} + \frac{\beta}{2} \norm{\nabla f\p{w_n} - \nabla f\p{w_{n-1}}}^2. \label{eq:non-smooth_f}
  \end{equation}
  From \cref{eq:non-smooth_f,eq:non-smooth_g} and the definition of $V_n$, we therefore obtain
  \begin{align}
    V_{n-1} - V_n
    &= f\p{w_{n-1}} - f\p{w_n} + g\p{x_n} - g\p{x_{n+1}} + \inpr{\nabla f\p{w_{n-1}}}{x_n - w_{n-1}} \nonumber \\
    &\qquad + \frac{1}{2 \beta} \norm{x_n - w_{n-1}}^2 - \inpr{\nabla f\p{w_n}}{x_{n+1} - w_n} - \frac{1}{2 \beta} \norm{x_{n+1} - w_n}^2 \nonumber \\
    &\geq \inpr{\nabla f\p{w_n} - \nabla f\p{w_{n-1}}}{w_{n-1} - x_n} + \frac{\beta}{2} \norm{\nabla f\p{w_n} - \nabla f\p{w_{n-1}}}^2 \nonumber \\
    &\qquad + \frac{1}{2 \beta} \norm{x_n - w_{n-1}}^2 + \frac{1}{\gamma_n} \norm{x_n - x_{n+1}}^2 - \frac{1}{2 \beta} \norm{x_{n+1} - w_n}^2 \nonumber \\
    &\qquad + \frac{1}{\beta} \inpr{\Delta x_n^1}{x_n - x_{n+1}} + \frac{1}{\gamma_n} \inpr{\Delta x_n^2}{x_n - x_{n+1}} \nonumber \\
    &= \frac{\beta}{2} \norm{\nabla f\p{w_n} - \nabla f\p{w_{n-1}} - \frac{1}{\beta} \p{x_n - w_{n-1}}}^2 \nonumber \\
    &\qquad + \frac{2 \beta - \gamma_n}{2 \beta \gamma_n} \norm{x_{n+1} - x_n - \frac{\beta}{2 \beta - \gamma_n} \Delta x_n^2}^2 \nonumber \\
    &\qquad 
    - \frac{1}{2 \beta} \norm{\Delta x_n^1}^2 - \frac{\beta}{2 \gamma_n \p{2 \beta - \gamma_n}} \norm{\Delta x_n^2}^2. \label{eq:non-smooth_lyapunov_difference}
  \end{align}
  Using \cref{eq:non-smooth_increment_bounds}, we then obtain the following inequality for all $n\geq 1$:
  \begin{multline}\label{eq:non-smooth_lyapunov_monotonicity}
    V_{n-1} + \frac{2 \beta - \gamma_{n-1}}{2 \beta \gamma_{n-1}} \norm{x_n - x_{n-1} - \frac{\beta}{2 \beta - \gamma_{n-1}} \Delta x_{n-1}^2}^2 
    \\ \shoveleft{\qquad
    \geq V_n + \frac{2 \beta - \gamma_n}{2 \beta \gamma_n} \norm{x_{n+1} - x_n - \frac{\beta}{2 \beta - \gamma_n} \Delta x_n^2}^2 
    }
    \\
    + \frac{\beta \p{1 - \tolno{n}}}{2} \norm{\nabla f\p{w_n} - \nabla f\p{w_{n-1}} - \frac{1}{\beta} \p{x_n - w_{n-1}}}^2 \\
    + \frac{\p{1 - \tolna{n}} \p{2 \beta - \gamma_{n-1}}}{2 \beta \gamma_{n-1}} \norm{x_n - x_{n-1} - \frac{\beta}{2 \beta - \gamma_{n-1}} \Delta x_{n-1}^2}^2.
  \end{multline}
  The claim that the sequence in \cref{eq:VnSeq} is monotonically non-increasing now follows from ignoring the two last terms in \cref{eq:non-smooth_lyapunov_monotonicity}, which are non-negative.
  This concludes the proof of the claims in \cref{item:thm:non-smooth_Lyapunov}.
  
  We next turn our attention to proving the claim in \cref{item:thm:non-smooth_summable}.
  First, let \(M \in \Real\) such that $f\p{x} + g\p{x} \geq M > -\infty$ holds for all $x\in\PrimS$ as $f + g$ is (by assumption) bounded from below.
  For fixed $N \geq 1$, consider the sum over $n = 1, \ldots, N$ for the terms in  \cref{eq:non-smooth_lyapunov_monotonicity}:
  \begin{multline*}
    V_0 + \frac{2 \beta - \gamma_0}{2 \beta \gamma_0} \norm{x_1 - x_0 - \frac{\beta}{2 \beta - \gamma_0} \Delta x_0^2}^2 
    \\ \shoveleft{\qquad
    \geq V_N + \frac{2 \beta - \gamma_N}{2 \beta \gamma_N} \norm{x_{N+1} - x_N - \frac{\beta}{2 \beta - \gamma_N} \Delta x_N^2}^2
    } 
    \\ \shoveleft{\qquad\qquad
    + \sum_{n = 1}^N \frac{\beta \p{1 - \tolno{n}}}{2} \norm{\nabla f\p{w_n} - \nabla f\p{w_{n-1}} - \frac{1}{\beta} \p{x_n - w_{n-1}}}^2 
    }
    \\ 
    + \sum_{n = 0}^{N-1} \frac{\p{1 - \tolna{n+1}} \p{2 \beta - \gamma_n}}{2 \beta \gamma_n} \norm{x_{n+1} - x_n - \frac{\beta}{2 \beta - \gamma_n} \Delta x_n^2}^2.
  \end{multline*}
  By the claim in \cref{item:thm:non-smooth_Lyapunov}, we know $V_N \geq f\p{x_{N+1}} + g\p{x_{N+1}} \geq M$, so 
  \begin{multline*}
    \sum_{n = 1}^N \frac{\beta \p{1 - \tolno{n}}}{2} \norm{\nabla f\p{w_n} - \nabla f\p{w_{n-1}} - \frac{1}{\beta} \p{x_n - w_{n-1}}}^2 \\
    + \sum_{n = 0}^{N-1} \frac{\p{1 - \tolna{n+1}} \p{2 \beta - \gamma_n}}{2 \beta \gamma_n} \norm{x_{n+1} - x_n - \frac{\beta}{2 \beta - \gamma_n} \Delta x_n^2}^2 \\
    \leq V_0 + \frac{2 \beta - \gamma_0}{2 \beta \gamma_0} \norm{x_1 - x_0 - \frac{\beta}{2 \beta - \gamma_0} \Delta x_0^2}^2 - M.
  \end{multline*}
  Next, by \cref{eq:non-smooth_stepsize_limsup,eq:non-smooth_radius_limsup}, there exist some $n_0 \geq 0$, $0 < \gamma_- \leq \gamma_+ < 2 \beta$, $0 \leq \tolna{+} < 1$ and $0 \leq \tolno{+} < 1$ such that
  \[
    \gamma_- \leq \gamma_n \leq \gamma_+, \qquad 0 \leq \tolna{n} \leq \tolna{+} \quad \text{and} \quad 0 \leq \tolno{n} \leq \tolno{+}
    \quad\text{for all $n\geq n_0$.}
  \]
  We therefore get
  \begin{multline*}
    \sum_{n = n_0}^N \norm{\nabla f\p{w_n} - \nabla f\p{w_{n-1}} - \frac{1}{\beta} \p{x_n - w_{n-1}}}^2 \\
    \leq \frac{2}{\beta \p{1 - \tolno{+}}} \p{V_0 + \frac{2 \beta - \gamma_0}{2 \beta \gamma_0} \norm{x_1 - x_0 - \frac{\beta}{2 \beta - \gamma_0} \Delta x_0^2}^2 - M}
  \end{multline*}
  and
  \begin{multline*}
    \sum_{n = n_0}^{N-1} \norm{x_{n+1} - x_n - \frac{\beta}{2 \beta - \gamma_n} \Delta x_n^2}^2 \\
    \leq \frac{2 \beta \gamma_+}{\p{1 - \tolna{+}} \p{2 \beta - \gamma_+}} \p{V_0 + \frac{2 \beta - \gamma_0}{2 \beta \gamma_0} \norm{x_1 - x_0 - \frac{\beta}{2 \beta - \gamma_0} \Delta x_0^2}^2 - M}.
  \end{multline*}
  Since the right-hand side is finite and independent of $N$ in both cases, we get
  \begin{align*}
    \sum_{n = n_0}^\infty \norm{\nabla f\p{w_n} - \nabla f\p{w_{n-1}} - \frac{1}{\beta} \p{x_n - w_{n-1}}}^2 &< +\infty, \\
    \sum_{n = n_0}^\infty \norm{x_{n+1} - x_n - \frac{\beta}{2 \beta - \gamma_n} \Delta x_n^2}^2 &< + \infty.
  \end{align*}
  Using \cref{eq:non-smooth_increment_bounds} and the fact that the convergence of the series does not depend on the (finite) number of elements with indices below $n_0$, we can immediately conclude that 
  \[
    \sum_{n = 0}^\infty \norm{\Delta x_n^1}^2 < +\infty \qquad \text{and} \qquad \sum_{n = 0}^\infty \norm{\Delta x_n^2}^2 < + \infty.
  \]
  The last assertion of \cref{item:thm:non-smooth_summable} follows from the observation below, which also concludes the proof of the claims in \cref{item:thm:non-smooth_summable}.
  \begin{multline*}
    \sum_{n = 0}^\infty \norm{x_{n+1} - x_n}^2 \\
    \leq 2\sum_{n = 0}^\infty \p{\norm{x_{n+1} - x_n - \frac{\beta}{2 \beta - \gamma_n} \Delta x_n^2}^2 + \frac{\beta^2}{\p{2 \beta - \gamma_n}^2} \norm{\Delta x_n^2}^2} < \infty.
  \end{multline*}

  We finally prove the claim in \cref{item:thm:non-smooth_cluster_ponts}.
  Let $x$ be a weak sequential cluster point of the sequence $\p{x_n}_{n \geq 0}$ with a sub-sequence $\p{n_k}_{k \geq 0}$ such that $x_{n_k} \weakto x$ as $k \to \infty$. By \cref{eq:non-smooth_inclusion}, we have
  \begin{equation}\label{eq:non-smooth_sub-sequence_inclusion}
    \frac{x_{n_k} - x_{n_k+1}}{\gamma_{n_k}} + \nabla f\p{x_{n_k+1}} - \nabla f\p{w_{n_k}} + \frac{\Delta x_{n_k}^1}{\beta} + \frac{\Delta x_{n_k}^2}{\gamma_{n_k}} \in \p{\nabla f + \partial g}\p{x_{n_k+1}}
  \end{equation}
  for all $k \geq 0$. The assertion in \cref{item:thm:non-smooth_summable} combined with the Lipschitz continuity of $\nabla f$ yields the following as $k \to \infty$:
  \begin{align*}
    0 &\leq \norm{\frac{x_{n_k} - x_{n_k + 1}}{\gamma_{n_k}}} \leq \frac{\norm{x_{n_k} - x_{n_k + 1}}}{\gamma_-} \to 0, \\
  0 &\leq \norm{\nabla f\p{x_{n_k + 1}} - \nabla f\p{w_{n_k}}} \leq \frac{1}{\beta} \p{\norm{x_{n_k + 1} - x_{n_k}} + \norm{\Delta x_{n_k}^1}} \to 0, \\
  0 &\leq \norm{\frac{\Delta x_{n_k}^2}{\gamma_{n_k}}} \leq \frac{1}{\gamma_-} \norm{\Delta x_{n_k}}^2 \to 0.
  \end{align*}
  Hence, the left-hand side of \cref{eq:non-smooth_sub-sequence_inclusion} converges to $0$ as $k \to \infty$. 
  Thus, by \cite[Proposition 20.38]{BauschkeCombettes:2017} we get $x_{n_k + 1} = \p{x_{n_k + 1} - x_{n_k}} + x_{n_k} \weakto x$, $0 \in \p{\nabla f + \partial g}\p{x}$, i.e.\@ $x$ is a minimizer of $f + g$. 
  The last statement follows from the fact that if $\norm{x_n} \not\to +\infty$, then there exists a bounded sub-sequence that has a weakly convergent sub-sequence whose weak limit is a minimizer of $f + g$.
  This concludes the proof of the claim in \cref{item:thm:non-smooth_cluster_ponts}.
\end{proof}

\begin{remark}
  if $g = 0$ on $\PrimS$, then \cref{eq:non-smooth_iteration} reads as 
  \[
  w_{n+1} = w_n - \gamma_n \nabla f\p{w_n} + \Delta x_{n+1}^1 + \frac{\gamma_n - \beta}{\beta} \Delta x_n^1 + \Delta x_n^2.
  \]
  This is equivalent with the algorithm in \cref{sec:smooth}, provided that $\gamma_n = \beta$ for all $n\geq 0$ and
  \begin{equation}\label{eq:SmoothAlgoIneq}
    \norm{\Delta x_n^2 + \Delta x_{n+1}^1} \leq \beta \norm{\nabla f\p{w_n}}.
  \end{equation}
  On the other hand, when $g = 0$ the condition in \cref{eq:non-smooth_increment_bounds} becomes 
  \begin{multline*}
    \frac{1}{2 \beta} \norm{\Delta x_n^1}^2 + \frac{\beta}{2 \gamma_n \p{2 \beta - \gamma_n}} \norm{\Delta x_n^2}^2 \\
    \leq \frac{\tolna{n} \p{2 \beta - \gamma_{n-1}}}{2 \beta \gamma_{n-1}} \norm{\gamma_{n-1} \nabla f\p{w_{n-1}} - \frac{\gamma_{n-1}}{\beta} \Delta x_{n-1}^1 - \frac{\beta - \gamma_{n-1}}{2 \beta - \gamma_{n-1}} \Delta x_{n-1}^2}^2 \\
    + \frac{\beta \tolno{n}}{2} \norm{\nabla f\p{w_n} - \frac{\beta - \gamma_{n-1}}{\beta} \nabla f\p{w_{n-1}} - \frac{\gamma_{n-1} - \beta}{\beta^2} \Delta x_{n-1}^1 - \frac{1}{\beta} \Delta x_{n-1}^2}^2.
  \end{multline*}
  In the case that $\gamma_n = \beta$ for all $n\geq 0$, this simplifies to the following inequality that is different from the one in \cref{eq:SmoothAlgoIneq}:
  \begin{multline*}
    \norm{\Delta x_n^1}^2 + \norm{\Delta x_n^2}^2 \\
    \leq \tolna{n} \norm{\beta \nabla f\p{w_{n-1}} - \Delta x_{n-1}^1}^2 + \tolno{n} \norm{\beta \nabla f\p{w_n} - \Delta x_{n-1}^2}^2.
  \end{multline*}
\end{remark}
\begin{remark}
    Setting $\Delta x_n^1 = \Delta x_n^2 = 0$ recovers the classical forward-backward algorithm \cite{Goldstein:1964}.
    Likewise, the \ac{FISTA} algorithm \cite{BeckTeboulle:2009} can be obtained by setting 
    \[ \Delta x_n^1 \defeq \frac{t_n-1}{t_{n+1}} (x_n-x_{n-1}) 
    \quad\text{and}\quad 
    \Delta x_n^2 \defeq \frac{\beta -\gamma_n}{\beta}\Delta x_n^1
    \quad\text{for all $n \geq 0$‚}
    \]
    with $( t_n )_{n=0}^{\infty}$ defined recursively as $t_{n+1} \defeq \frac{1}{2}(1 + \sqrt{1 + 4t_{n}^2})$ for $n \geq 1$ and $t_0\defeq1$.
\end{remark}

\section{Deep-learning accelerated optimization}\label{sec:methods}
\Cref{sec:smooth,sec:non-smooth} introduced deviation-based iterative schemes for solving smooth and non-smooth minimization problems that contain many traditional iterative schemes as special cases.
These schemes rely on a rule for selecting deviations as iterates progress.
Selection rules that give rise to convergent sequences will be referred to as \emph{convergent} and \cref{thm:smooth_main_convergence,thm:non-smooth_main} provide criteria that imply convergence. 

A natural next step is to look for a selection rule that has \enquote{optimal} convergence. 
This requires one to formalize the notion of \enquote{optimal}.
Such a selection rule will most likely be problem dependent, i.e.\@ we cannot expect to have a single selection rule that is optimal for all minimization problems.
One approach is to search for the optimization method that is optimal \emph{on average} when applied to a specific class of minimization problems.

Based on the above, we adopt principles from statistical decision theory to learn an optimal convergent selection rule from suitable example data.
It is clearly unfeasible to search over all possible convergent selection rules, so attention is restricted to those that are encoded by a (deep) neural network.
A key issue is to select a neural network architecture that yields a convergent selection rule.
The following sections provide details of this approach.

\subsection{Optimization parametrized by an additional input}\label{ssec:ParamOpt}
The first task is to specify the class of minimization problems. 
Here we simply consider minimizing an objective function on $\PrimS$ that is parametrized by an additional input $y \in \OptParamSp$, i.e.
\begin{equation}\label{eq:FminParam0}
    \min_{x \in \PrimS} F_y\p{x} \quad\text{for $y \in \OptParamSp$.}
\end{equation}
$F_y \colon \PrimS \to \RealExt$ above is a proper, convex and lower semi-continuous function and $\OptParamSp$ is the set containing the additional input for the objective, i.e., the input data. For simplicity, we will assume that $\PrimS$ is finite-dimensional and that each $F_y$ is strictly convex and coercive so that~\cref{cor:non-smooth,cor:smooth} are applicable and provide conditions for strong convergence of the respective iteration sequences.
Such problems naturally arise in many applications, e.g., when using variational regularization to solve ill-posed inverse problems (see \cref{sec:applications}).

Next is to formalize the notion of \emph{optimal convergence} for a optimization scheme that aims to solve problems of the form in \cref{eq:FminParam0}, which in turn dictates how to train the neural network.
Computational feasibility was earlier stated as our main concern, so we consider a setting where the computational budget is limited in terms of a fixed number of iterations $N$, other possible choices are discussed in \cref{rem:discussion}.
Consider an iterative scheme with parameters $\theta \in \Theta$ applied for $N \in \mathbb{N}$ iterations when the initial point $x_0 \in \PrimS$ is fixed, this defines a function of the input $y \in \OptParamSp$ for the objective:
\[
\OptSolver_{N, \theta} \colon \OptParamSp \to \PrimS
\]

Next, the aforementioned optimality means we look for the optimization method, which on average (as $y \in \OptParamSp$ varies) renders an iterate after $N$ steps that is as close as possible to the true minimal value of the objective in \cref{eq:FminParam0}.
Stated mathematically, let $Y$ denote a $\OptParamSp$-valued random variable generating $y \in \OptParamSp$. 
For optimality we look for $\OptSolver_{N,\hat{\theta}}(\,\cdot\,) \colon \OptParamSp \to \PrimS$ where 
\begin{equation}\label{eq:TrainingStat}
\hat{\theta} \in \argmin_{\theta \in \Theta} 
  \mathbb{E}_Y \Bigl[ 
    F_Y\bigl(\OptSolver_{N,\theta}(Y)\bigr)
  \Bigr].
\end{equation}
In practice, the distribution of $Y$ is not available, instead one has  samples $\{ y_i \}_{i=1}^m \subset \OptParamSp$ generated by $Y$.
We therefore consider its empirical counterpart:
\begin{equation}\label{eq:TrainingEmp}
\hat{\theta} \in \argmin_{\theta \in \Theta} 
  \frac{1}{m}\sum_{i=1}^m F_{y_i}\bigl(\OptSolver_{N,\theta}(y_i)\bigr).
\end{equation}

The unsupervised learning problem\footnote{This is unsupervised learning since training data only consists of $y_i$'s and not the $x^*_i(y)$'s that minimize $x \mapsto F_{y_i}(x)$.} \cref{eq:TrainingEmp} results in a scheme that provides an output which is on average closest to minimizing the objective in \cref{eq:FminParam0} after $N$ iterations.
Combining this with convergence means ensuring the following holds for all $y\in \OptParamSp$:
\begin{equation}\label{eq:Convergence}
  \lim_{n \to \infty} \OptSolver_{n,\hat{\theta}}(y)
  = x^*(y) 
\end{equation}
where $x^*(y) \in \PrimS$ solves \cref{eq:FminParam0}. Hence, a solver given by $\hat{\theta}$ that is trained to be optimal for $N$ iterates will converge to a solution of \cref{eq:FminParam0} if we run more iterations, even when it is not trained for this scenario.

The next two sections provide further details on how to define a deviation-based iterative solver $\OptSolver_{N,\theta} \colon \OptParamSp \to \PrimS$ in terms of a parametrized selection rule $\Psi_{\theta}$ for $\theta \in \Theta$.
The optimal scheme is obtained by training as in \cref{eq:TrainingEmp} and convergence is ensured by an appropriate choice of neural network architecture. 

\begin{remark}\label{rem:discussion}

The topic of how to define \enquote{optimal convergence} is of high importance to optimization, but the choice becomes especially apparent in learning-based schemes. In the classical optimization literature \cite{GonzagaKarasRossetto:2013,Neumaier:2016,AhookhoshNeumaier:2017a,AhookhoshNeumaier:2017b} the most widely used definition is upper bounds on the worst-case performance as the number of iterations go to infinity. For learning this is not an attractive definition since it is hard to evaluate computationally. In this work we chose to optimize the average function value after a fixed number of steps, but it would also be possible to optimise e.g. the average number of steps needed for convergence.

\end{remark}

\subsection{Smooth convex objective}\label{ssec:smooth_method}
Assume that $F_y \colon \PrimS \to \Real$ in \cref{eq:FminParam0} is differentiable with a $\beta^{-1}$-Lipschitz continuous gradient.
Next, consider iterates $\p{x_n\p{y}}_{n\geq 0} \subset \PrimS$ generated by an iterative scheme for solving \cref{eq:FminParam0}.
The mapping $\OptSolver_{N,\theta} \colon \OptParamSp \to \PrimS$ is defined as the $N$-th term in that sequence, i.e.\@ $\OptSolver_{N,\theta}(y) = x_N(y)$.

A natural approach is to consider the deviation-based scheme in \cref{eq:smooth_iteration}, which, when applied to the $\OptParamSp$-parametrized minimization in \cref{eq:FminParam0}, reads as 
\[
  x_{n+1}(y) \defeq x_n(y) - \beta\Bigl( \nabla F_y\bigl(x_n(y)\bigr) + \Delta x_n(y)\Bigr)
  \quad\text{for all $n\geq 0$.}
\]
The selection rule for setting the deviations $\Delta x_n(y)$ above is defined as
\begin{equation}\label{eq:DNNdeviation}
  \Delta x_n(y) \defeq \Psi_{\theta}\Bigl( x_n(y), \nabla F_y\bigl(x_n(y)\bigr), \Delta x_{n-1}(y)\Bigr).
\end{equation}
where $\Psi_{\theta} \colon \PrimS \times \PrimS \times \PrimS \to \PrimS$ is an updating function that ensures that the resulting deviations satisfy the conditions in \cref{eq:smooth_limsup}.
The above completes the definition of $\OptSolver_{N,x_0} \colon \OptParamSp \times \Theta \to \PrimS$ and the optimal solver can be obtained from example data by training as in \cref{eq:TrainingEmp}.

The updating function $\Psi_{\theta}$ will be given by a (deep) neural network that is parametrized by $\theta \in \Theta$.
We construct the neural network $\Psi_{\theta}$ by composing a conventional neural network $\Tilde{\Psi}_{\theta} \colon \PrimS \times \PrimS \times \PrimS \to \PrimS$ with a normalizing function $\normalize \colon \PrimS \times \PrimS \to \PrimS$, which ensures that the outcome is a feasible deviation. 
More precisely, we define 
\begin{equation*}
  \Psi_{\theta}\Bigl( x_n(y), \nabla F_y\bigl(x_n(y)\bigr), \Delta x_{n-1}(y)\Bigr)
\defeq
    \normalize\Bigl(h_n,\nabla F_y\bigl(x_n(y)\bigr)\Bigr)
\end{equation*}
where
\begin{align}
  & \normalize\Bigl(h_n,\nabla F_y\bigl(x_n(y)\bigr)\Bigr) 
  \defeq \tols{}  \frac{h_n}{\sqrt{\norm{h_n}^2 + 1}}  
  \Bigl\| \nabla F_y\bigl(x_n(y)\bigr) \Bigr\|,
  \label{eq:smooth_normalization}
\\[0.75em] 
   & h_n \defeq \Tilde{\Psi}_{\theta}\Bigl(x_n(y), \nabla F_y\bigl(x_n(y)\bigr), \Delta x_{n-1}(y)\Bigr).
\label{eq:standard_nn}
\end{align}

Here $\tols{} \in [0, 1)$ controls the size of the learned deviations within the feasible region. 
\Cref{sec:Results} presents performance for different values of $\tols{}$ in ablation experiments.
Regarding the choice of $\normalize$, the essential property is that its range is contained in the ball that guarantees that the iterative scheme is convergent as formalized in \cref{eq:smooth_limsup} with $\tols{n} = \tols{}$. 
Thus, one can essentially replace $\normalize$ with any transformation with this property.

\begin{remark}
In \cref{eq:DNNdeviation} we use a current iterate, gradient at that iterate, and previous deviation as input for the network $\Psi_{\theta}$ that defines the selection rule for the deviations.
This choice is largely heuristic and often (as in the problem discussed in \cref{ssec:smooth_practical}), the objective function naturally consists of two parts $F_y = f_y + g_y$.
Here one has the option to evaluate gradients of both parts separately instead of taking the gradient of the objective.
This corresponds to a selection rule of the form 
\begin{equation}\label{eq:smooth_updates}
\Delta x_n(y) \defeq \Psi_{\theta}\Bigl(
  x_{n}(y), 
  \nabla f_y\bigl(x_{n}(y)\bigr), 
  \nabla g_y\bigl(x_{n}(y)\bigr), 
  \Delta x_{n-1}(y)
  \Bigr)
\end{equation}
for some neural network $\Psi_{\theta} \colon \PrimS \times \PrimS \times \PrimS \times \PrimS \to \PrimS$.
Moreover, it is possible to incorporate information from previous iterations even though we did not pursue this approach in our experiments.
\end{remark}

\subsection{Non-smooth convex  objective}\label{ssec:non-smooth_method}
Here $F_y \colon \PrimS \to \Real$ in \cref{eq:FminParam0} is convex, but not necessarily differentiable. 
We specialize to the case $F_y = f_y + g_y$ where $f_y \colon \PrimS \to \Real$ and $g_y \colon \PrimS \to \RealExt$ are convex but only $f_y$ is differentiable and $\nabla f_y \colon \PrimS \to \PrimS$ is $\beta^{-1}$-Lipschitz continuous, we seek to solve 
\begin{equation}\label{eq:NonSmoothOptimSum}
    \min_{x \in \PrimS} f_y(x) + g_y(x)
\end{equation} 

A natural approach is to consider the deviation-based scheme in \cref{eq:non-smooth_iteration_grad,eq:non-smooth_iteration_prox} for solving the above problem.
Similarly as in~\cref{ssec:smooth_method}, the selection rules for the updates $\Delta x_n^1(y)$ and $\Delta x_n^2(y)$ are given by neural networks $\Psi^1$ and $\Psi^2$, i.e.
\begin{subequations}\label{eq:non-smooth_updates}
\begin{align}
\Delta x_n^1(y) 
  &\defeq \Psi^1_{\theta^1}\Bigl( 
       x_n(y), \nabla f_y\bigl(w_{n-1}(y)\bigr), 
       \Delta x_{n-1}^1(y) 
     \Bigr) \\
\Delta x_n^2(y) 
  &\defeq \Psi^2_{\theta^2}\Bigl(
       x_n(y), \nabla f_y\bigl(w_{n-1}(y)\bigr), 
       \Delta x_{n-1}^2(y), \Delta x_{n}^1(y)
     \Bigr).
\end{align}
\end{subequations}
The optimal solver is now obtained by training $\OptSolver_{N,x_0} \colon \OptParamSp \times \Theta \to \PrimS$ as in \cref{eq:TrainingEmp}. 
To ensure convergence as in \cref{eq:Convergence}, the resulting learned updates need to satisfy  \cref{eq:non-smooth_increment_bounds} in  
\cref{thm:non-smooth_main} according to \cref{cor:non-smooth}.
Similar to the smooth case in \cref{ssec:smooth_method}, we achieve this by normalizing the output of the last hidden layers $h_n^1$ and $h_n^2$ in the two neural networks:
\begin{subequations}\label{eq:non-smooth_normalization}
\begin{align}
  \Delta x_n^1 &= \sqrt{\frac{ \tolna{} (2\beta- \gamma)}{\gamma }}  \frac{h_n^1}{\sqrt{\|h_n^1\|^2 +1}}  \norm{x_n - x_{n-1} - \frac{\beta}{2\beta - \gamma} \Delta x_{n-1}^2}, \\
  \Delta x_n^2 &= \sqrt{\gamma(2\beta - \gamma)\tolno{}}  \frac{h_n^2}{\sqrt{\|h_n^2\|^2 + 1}}   \norm{\nabla f_y\p{w_n} - \nabla f_y\p{w_{n-1}} - \frac{1}{\beta} \p{x_n - w_{n-1}}}.
\end{align}
\end{subequations}
As before, $\tolna{}, \tolno{} \in [0,1)$ control the sizes of the learned deviations in the feasible region.

\section{Numerical experiments}\label{sec:applications}
This section evaluates performance of the proposed learned optimizers for solving minimization problems arising in tomographic image reconstruction. 

\subsection{Variational regularization in tomographic imaging}
Tomography is a collection of techniques that seek to visualize the interior structure of an object by probing it with penetrating particles/waves from different directions.  
The most well-known example is \acf{CT} in medical imaging.
A patient is here scanned by X-rays from different directions and the aim is to recover a 2D/3D image of the interior anatomy.

Tomographic imaging inevitably leads to an inverse problem since the interior structure $x \in \PrimS$ (2D/3D image) is only indirectly observed through data $y \in \OptParamSp$. 
This can be formalized as solving an equation
\begin{equation}\label{eq:InvProb}
  y = A x + \mathrm{noise} 
\end{equation}  
where $A \colon \PrimS \to \OptParamSp$ models how data is generated in absence of noise.
$A$ can be taken as a linear operator (ray transform) if one adopts a simplified model for how X-rays interact with tissue and properly pre-processes raw sensor data. 

Tomographic imaging problems are unfortunately often ill-posed.
This means that a solution procedure that seeks to maximize data consistency, e.g.\@ by minimizing $x \mapsto \|A x -y\|^2$, will be unstable. 
Variational regularization addresses these issues by adjusting the need for data consistency against the need to suppress unwanted features (over-fitting).
Stated formally, instead of trying to solve \cref{eq:InvProb}, one solves 
\begin{equation}\label{eq:VarReg}
    \hat{x} \in \argmin_{x \in \PrimS} \Big\{ \|A x -y\|^2 + \RegFunc_{\lambda}(x) \Big\}.
\end{equation}
Here, $\RegFunc_{\lambda} \colon \PrimS \to \Real$ (regularization functional) stabilizes the recovery procedure, typically by enforcing certain regularity. 
\begin{remark}
Much effort over the last three decades has been devoted to determining mathematical properties of solutions to \cref{eq:VarReg} for various choices of $\RegFunc_{\lambda}$.
Such analysis is typically performed in a non-discretized setting where both $x$ (2D/3D image) and $y$ (data)  resides in infinite-dimensional function spaces, see \cite{Scherzer:2009aa,Burger:2013aa,Benning:2018aa} for extensive surveys.
\end{remark}

\paragraph{Characteristics of \cref{eq:VarReg}}
The minimization in \cref{eq:VarReg} is of the form in  \cref{eq:FminParam0} with tomographic data $y \in \OptParamSp$ as the additional variable. 
In our case, the objective in \cref{eq:VarReg} has a natural decomposition $F_y = f_y + g$ where 
\[ 
f_y(x) \defeq \|A x -y\|^2 
\quad\text{and}\quad
g(x) \defeq \RegFunc_{\lambda}(x).
\]
Note here that $f_y \colon \PrimS \to \Real$ is differentiable whenever $A$ is differentiable, which in particular is the case when $A$ is linear.
Properties of $g \colon \PrimS \to \Real$ will depend on the choice of regularization functional, but it will be proper, convex and lower semi-continuous.

Next, in tomographic imaging, both $x$ and $y$ are high-dimensional arrays after discretization.
In the 2D setting that is considered in \cref{sec:Results}, $x$ represents $512 \times 512$ pixel 2D coronal image slices.
These are obtained from normal-dose \ac{CT} scans of the human abdomen provided by Mayo Clinic for the AAPM Low Dose CT Grand Challenge \cite{McCollough:2017aa}
and used as input for simulating X-ray parallel-beam tomographic data.
The latter is from 1000 equidistributed source positions  and 1000 detector elements with $5\%$ additive Gaussian noise, i.e.\@ data $y$ is a $10^6$-dimensional array.
In summary, even though we consider a 2D setting, \cref{eq:VarReg} is a large-scale optimization problem. 

An additional aspect to consider is that most clinical imaging studies are performed in a time-critical setting. 
An informal rule of thumb for \ac{CT} imaging is that reconstructed images need to be available within 5~minutes after scanning\footnote{This would exclude usages of \ac{CT} in trauma where reconstructions need to be available instantaneously.}.
Hence, applicability of variational regularization in clinical \ac{CT} rests to large extent on the ability to (approximately) solve \cref{eq:VarReg} within this time-frame, so one can typically only afford about 10  evaluations of $A$ and $A^{\top}$.

\subsection{Smoothed total-variation regularization} \label{ssec:smooth_practical}
Total-variation regularization corresponds to choosing  $\RegFunc_{\lambda}(x) \defeq \lambda \|\nabla x\|_1$ in \cref{eq:VarReg} for some $\lambda>0$.
Using such a regularization functional is popular since it preserves edges \cite{rudin1992nonlinear,Caselles:2015aa}.
When $x$ is discretized, the gradient $\nabla$ is typically replaced by the linear operator $D \colon \PrimS \to \PrimS \times \PrimS$ that computes differences in pixel values along vertical and horizontal directions.

A computational drawback with total variation regularization is that it results in minimizing a non-smooth objective. 
A common remedy is therefore to replace the 1-norm with the Huber function, i.e., choosing $\RegFunc_{\lambda}(x) \defeq \lambda H_{\delta}(D x)$ where $\delta > 0$, $x = \p{x_i}_{i \in I}$, and
\begin{equation}
    H_{\delta}(x) \defeq \sum_{i \in I} h_{\delta}(x_i), 
    \quad\text{with}\quad
    h_{\delta}(x_i) \defeq \begin{cases} \dfrac{1}{2\delta} x_i^2 & \quad \text{if } |x_i| < \delta \\[0.5em]
    |x_i| - \dfrac{\delta}{2} & \quad \text{otherwise.}
    \end{cases}
\end{equation}
This results in a smooth objective as \cref{eq:VarReg} now reads as 
\begin{equation}\label{eq:smooth_problem}
    \hat{x} \defeq \argmin_{x\in\PrimS} \|Ax -y\|^2 + \lambda H_{\delta}(D x),
\end{equation}
which corresponds to \cref{eq:FminParam0} with $F_y(x) = f_y(x) + g(x)$ where
\begin{equation}
 f_y(x) \defeq \norm{Ax - y}^2
 \quad\text{and}\quad
 g(x) \defeq \lambda H_{\delta}(Dx),
\end{equation}
so
\begin{equation}
    \nabla F_y(x) = \nabla f_y(x) + \nabla g(x) = 2A^{\top}(Ax-y) + \lambda D^\top \nabla(H)_{\delta} (Dx).
\end{equation}
One may furthermore normalize the linear operator $A$, i.e.,  $\norm{A} = 1$.
Then, $\nabla F_y$ is Lipschitz continuous with constant $\beta^{-1} = 2\|A\|^2  + \frac{\lambda}{\delta}\|D\|^2= 2 + 8 \lambda/\delta$. 

The tests in \cref{sec:Results} compare various deviation-based schemes for solving \cref{eq:smooth_problem} with $\delta = 0.01$ (Huber parameter) and $\lambda = 0.0015$ (regularization parameter).
The performance of the learned deviation-based scheme in \cref{ssec:smooth_method} is evaluated for different values of the normalization parameter $\tols{}$ in \cref{eq:smooth_normalization}. 

\subsection{Sparsity-promoting regularization} \label{ssec:non-smooth_practical}
Sparsity is another powerful method for regularizing ill-posed inverse problems, and in particular those that are under-sampled \cite{Fornasier:2015aa}.
A typical choice as regularizing functional in \cref{eq:VarReg} is to pick $\RegFunc_{\lambda}(x) = \lambda \|Wx\|_1$ where $\lambda>0$, so we end up with the following non-smooth minimization problem:
\begin{equation}\label{eq:problem}
    \hat{x} \defeq \argmin_x \|Ax -y\|^2 + \lambda \|Wx\|_1.
\end{equation}
In the above, $x \mapsto Wx$ is a compression that preserves essential features (sparsifying transform).
For the tests in \cref{sec:Results}, we select it as an orthogonal wavelet transform given by Symlets with filter length 10 and 5 scale levels. The regularization parameter is set to $\lambda = 0.0005$.

The minimization in \cref{eq:problem} corresponds to \cref{eq:FminParam0} with a non-smooth objective $F_y(x) = f_y(x) + g(x)$ where
\begin{equation}
 f_y(x) \defeq \norm{Ax - y}^2
 \quad\text{and}\quad
 g(x) \defeq \|Wx\|_1.
\end{equation}
As in \cref{ssec:smooth_practical}, $A$ is normalized, so $\|A\| = 1$, and $\nabla f_y(x) = 2A^{\top}(Ax-y)$ is Lipschitz continuous with constant $\beta^{-1} = 2$. 
Hence, $f_y$ and $g$ satisfy the assumptions in \cref{sec:non-smooth} and one can solve the problem by using the method proposed in  \cref{ssec:non-smooth_method}.
To proceed, we use \cite[Proposition 24.14] {BauschkeCombettes:2017} to compute the proximal in \cref{eq:non-smooth_iteration_prox}:
\begin{equation}
    \Prox{\lambda\gamma_n}{g}(x) = x + \frac{1}{\mu} W^\top \left( \Prox{\mu\lambda\gamma_n}{\| \cdot \|_1}(Wx) - Wx \right)
\end{equation}
where $\Prox{\gamma}{\| \cdot \|_1} (x) = \operatorname{sign}(x) \cdot \max(|x| - \gamma, 0)$ is a proximal of $l_1$ norm and $\mu$ is such that $\mu \operatorname{Id}= W W^\top$.

The tests in \cref{sec:Results} compare various deviation-based schemes for solving \cref{eq:problem} with gradient descent step-size $\gamma_n = 1 / (2 \|A\|^2)= 0.5$.
The performance of the learned deviation-based scheme in \cref{ssec:non-smooth_method} is evaluated for different values of the normalization parameters $\tolna{} = \tolno{} = \tols{}$ in \cref{eq:non-smooth_normalization}.

\subsection{Baseline methods}
In the case of smooth optimization (i.e.\@ solving \cref{eq:smooth_problem}), we compare our method to the steepest gradient descent and Nesterov's accelerated scheme \cite{Nesterov:1983}. For the steepest gradient descent, the updates are the same as in \cref{eq:smooth_iteration}, but with $\Delta x_n = 0$:
\begin{equation}
    x_{n+1} \defeq x_n - \beta\nabla f\p{x_n}
\end{equation}
In Nesterov's method the updates are:
\begin{subequations}\label{eq:nesterov_iteration}
\begin{align}
    x_{n} &\defeq w_n - \beta \nabla f\p{w_n}, \\
    t_{n+1} &\defeq \frac{1 + \sqrt{(1+4t_{n}^2)}}{2}, \\
    w_{n+1} &\defeq x_{n} + \frac{t_{n}-1}{t_{n+1}}(x_{n} - x_{n-1}),
\end{align}
\end{subequations}
where $t_0\defeq1$.

In the non-smooth case (i.e.\@ solving \cref{eq:problem}), we compare our method against
\ac{ISTA} \cite{Goldstein:1964} and \ac{FISTA} \cite{BeckTeboulle:2009}. 
Iterates in \ac{ISTA} are updated according to
\begin{subequations}\label{eq:ista_iteration}
\begin{align}
  x_{n+1} &\defeq \Prox{\gamma_n}{g}\p{x_n - \gamma_n \nabla f\p{x_n}}. \label{eq:ista_iteration_prox}
\end{align}
\end{subequations}
This is equivalent to updates in  \cref{eq:non-smooth_iteration_grad,eq:non-smooth_iteration_prox} with $\Delta x_n^1 = \Delta x_n^2 = 0$.
In \ac{FISTA}, updates are given as
\begin{subequations}\label{eq:fista_iteration}
\begin{align}
    x_{n} &\defeq \Prox{\gamma_n}{g}\p{w_n - \gamma_n \nabla f\p{w_n}}, \\
    t_{n+1} &\defeq \frac{1 + \sqrt{(1+4t_{n}^2)}}{2}, \\
    w_{n+1} &\defeq x_{n} + \frac{t_{n}-1}{t_{n+1}}(x_{n} - x_{n-1}),
\end{align}
\end{subequations}
where $t_0\defeq1$.
Finally, we do not provide a comparison to \ac{LISTA} \cite{GregorLecun:2010}, because an input of dimension $n$ it requires training $n^2$ parameters. 
The memory requirements for this method are prohibitively large for the examples in \cref{sec:Results} since $n= 512\times 512$.

\subsection{Implementation}
The architecture for the neural networks in \cref{eq:DNNdeviation,eq:non-smooth_updates} is similar to the networks representing the learned updates in \cite{Adler:2018aa}. 
Thus, it consists of three convolutional layers (two hidden and one output layer) with $3\times 3$ convolutional kernels. 
The hidden layers have 32 filters that are followed by instance normalization and leaky ReLU as an activation function. 
Instead of performing explicit pre-processing of the input data, we use additional instance normalization before the hidden layers. 
This results in a network with a total of $1 \cdot 10 \cdot 32 + 32 \cdot 10 \cdot 32 + 32 \cdot 10 \cdot 1 =10\,880$ free parameters that need to be set during training.

To train the networks, we use the unsupervised loss function $F(x_N)$, where $N$ is the total number of iterations. However, during the evaluation, we run much more iterations to verify the convergence, despite the fact that the network was not trained for this scenario. During the training, we sample $N$ uniformly between 10 and 20. Randomizing the number of iterations is necessary to ensure a smooth transition between the scenario for which the model is trained and for which it is not. Since the loss function $F(x_n)$ is not guaranteed to decrease monotonically, we observed that fixing the number of iterations to 10 during the training results in a slight increase of the loss after 10 iterations, before it starts to decrease again.

The neural networks are implemented in Tensorflow \cite{abadi2016tensorflow} and trained for $10^5$ iterations using the Adam optimizer with learning rate $10^{-3}$ and batch size 1. This takes about three days on a workstation with a single GeForce RTX~2080~Ti GPU. However, for most presented models $3 \cdot 10^4$ training iterations are sufficient.
Tomography-related operations necessary for evaluating the data consistency along with bindings to TensorFlow are implemented in ODL \cite{odl} with ASTRA \cite{van2016fast} as the computational back-end. We make the code publicly available\footnote{\href{https://github.com/JevgenijaAksjonova/Deep-Optimization}{github.com/JevgenijaAksjonova/Deep-Optimization}}.

\subsection{Results}\label{sec:Results}
All the solution algorithms initialize iterations to zero image $x_0 = 0$. 
Next, all deviation-based schemes with learned selection rules have been trained on data from nine of ten patients. The 210 images corresponding to the last patient are used for testing.

\subsubsection{Smoothed total variation regularization (smooth optimization)}
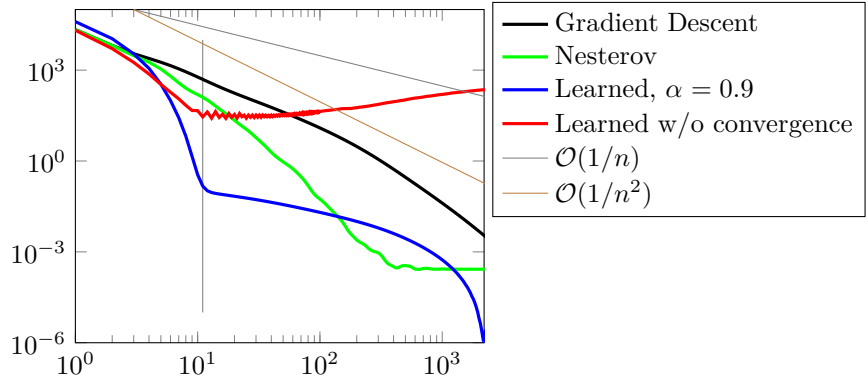
\begin{figure}
  \centering
  \begin{tikzpicture}
      \begin{loglogaxis}
      [
      ymin=0.000001,
      ymax=100000,
      xmin=1,
      xmax=2200,
      height=6cm,
      legend style = {anchor=west, at={(1.02, 0.70)}},
      legend cell align={left},
      no markers,
      cycle list={
        {black, solid, very thick},
        {green, solid, very thick},
        {blue, solid, very thick},
        {red, solid, very thick},
        {gray, solid},
        {brown, solid}}
      ]
      \draw[gray] (11, .00001) -- (11, 10000);
      \addplot table {plot_data/smooth_case/steep_desc.tab};
      \addplot table {plot_data/smooth_case/nesterov.tab};
      \addplot table {plot_data/smooth_case/learned_a.9.tab};
      \addplot table {plot_data/smooth_case/learned_div.tab};
      \addplot (3, 100000) -- (30, 10000) -- (300, 1000) -- (3000, 100);
      \addplot (3, 100000) -- (30, 1000) -- (300, 10) -- (3000, .1);
      \legend{Gradient Descent\\ Nesterov \\ Learned, $\tols{}=0.9$ \\ Learned w/o convergence\\ $\mathcal{O}(1/n)$\\ $\mathcal{O}(1/n^2)$\\}
      \end{loglogaxis}
  \end{tikzpicture}
  \caption{Plot of $F\p{x_n} - F^*$ as a function of the iteration number for gradient descent, Nesterov's accelerated method, the learned scheme in \cref{ssec:smooth_method} with $\tols{}=0.9$, and a variant of latter without the normalization given by~\cref{eq:smooth_normalization}.}
  \label{fig:smooth_results}
\end{figure}

In \cref{fig:smooth_results} we visualize the performance of different methods averaged over the set of test images. We identify the minimal objective value $F^*$ achieved among all the methods and plot the difference $F(x_n) - F^*$ to the minimal value in logarithmic scale depending on the iteration.

We present results for the following methods: the gradient descent, Nesterov's accelerated method, and the learned scheme described in \cref{ssec:smooth_method} with the normalization parameter $\tols{} = 0.9$. For comparison, we also include a variant of the learned scheme without the normalization given by~\cref{eq:smooth_normalization}.

We use 32-bit single-precision floating-point numbers in our experiments (that represents about 7 decimal digits) and we average over 210 images. The final loss for each image is in most cases a 3 digit number. Therefore we expect the results be precise only up to 5 digits, which is up to and including $10^{-2}$. 

As we can see, learned schemes perform better than the baseline methods during the first 10 iterations (marked with a grey vertical line). This result is expected since it corresponds to the objective set during the training. We also evaluate a variant of the learned scheme that uses a network without the normalization in \cref{eq:smooth_normalization} and, hence, does not have convergence guarantees. Not only does it achieve a worse performance during the first 10 iterations, but it also quickly diverges afterwards.

\begin{figure}
  \centering
  \begin{tikzpicture}
      \begin{loglogaxis}
      [
      ymin=0.000001,
      ymax=100000,
      xmin=1,
      xmax=2200,
      height=6cm,
      legend style = {anchor=west, at={(1.02, 0.70)}},
      legend cell align={left},
      no markers,
      cycle list={
        {black, solid, very thick},
        {green, solid, very thick},
        {blue, solid, very thick},
        {red, solid, very thick},
        {gray, solid},
        {brown, solid}}
      ]
      \draw[gray] (11, .00001) -- (11, 10000);
      \addplot table {plot_data/smooth_case/steep_desc.tab};
      \addplot table {plot_data/smooth_case/learned_e.001.tab};
      \addplot table {plot_data/smooth_case/learned_a.9.tab};
      \addplot table {plot_data/smooth_case/learned_a.5.tab};
      \addplot (3, 100000) -- (30, 10000) -- (300, 1000) -- (3000, 100);
      \addplot (3, 100000) -- (30, 1000) -- (300, 10) -- (3000, .1);
      \legend{Gradient Descent \\ Learned,  $\tols{}=0.999$\\ Learned, $\tols{}=0.9$\\ Learned, $\tols{}=0.5$\\ $\mathcal{O}(1/n)$\\ $\mathcal{O}(1/n^2)$\\}
      \end{loglogaxis}
  \end{tikzpicture}
  \caption{Plot of $F\p{x_n} - F^*$ as a function of the iteration number for gradient descent and the learned scheme in \cref{ssec:smooth_method} that enforces convergence through  \cref{eq:smooth_normalization} for different values of $\tols{}$.}
  \label{fig:smooth_ablation}
\end{figure}
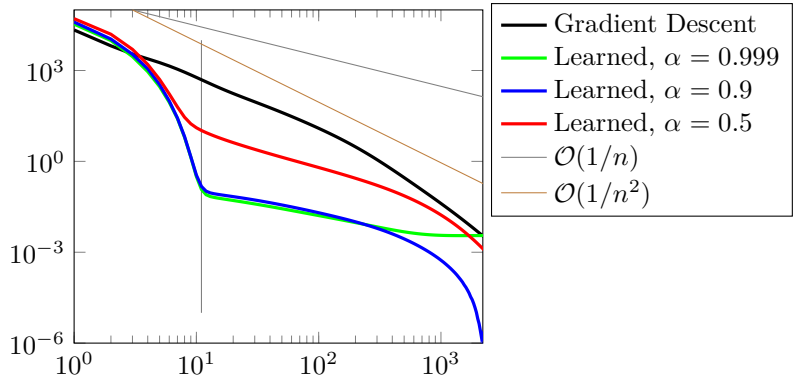

In~\cref{fig:smooth_ablation}, we evaluate the performance of the learned scheme for different choices of the normalization parameter $\tols{}$, which controls the size of the learned deviation $\Delta x_n$. The largest value $\tols{} =0.999$ leads to the best performance during the first 10 iterations, however, $\tols{} = 0.9$ performs almost as good and converges much faster.  Decreasing $\tols{}$ further makes the method  approach the steepest gradient descent, by enforcing convergence and sacrificing performance during the first 10 iterations.

\subsubsection{Sparsity-promoting regularization (non-smooth optimization)}
\begin{figure}
  \centering
  \begin{tikzpicture}
    \begin{loglogaxis}
      [
      ymin=0.000001,
      ymax=100000,
      xmin=1,
      xmax=650,
      height=6cm,
      legend style = {anchor=west, at={(1.02, 0.70)}},
      legend cell align={left},
      no markers,
      cycle list={
        {black, solid, very thick},
        {green, solid, very thick},
        {blue, solid, very thick},
        {red, solid, very thick},
        {gray, solid},
        {brown, solid}}
      ]
      \draw[gray] (11, .00001) -- (11, 10000);
      \addplot table {plot_data/nonsmooth_case/ista.tab};
      \addplot table {plot_data/nonsmooth_case/fista.tab};
      \addplot table {plot_data/nonsmooth_case/learned_dx2_a.5.tab};
      \addplot table {plot_data/nonsmooth_case/learned_dx2_div.tab};
      \addplot (3, 100000) -- (30, 10000) -- (300, 1000) -- (3000, 100);
      \addplot (3, 100000) -- (30, 1000) -- (300, 10) -- (3000, .1);
      \legend{ISTA\\ FISTA \\ Learned, $\tols{}=0.5$\\ Learned w/o convergence \\  $\mathcal{O}(1/n)$\\ $\mathcal{O}(1/n^2)$\\}
    \end{loglogaxis}
  \end{tikzpicture}

  \caption{Plot of $F\p{x_n} - F^*$ as a function of the iteration number for \ac{ISTA} \cite{Goldstein:1964}, \ac{FISTA} \cite{BeckTeboulle:2009}, the learned scheme in \cref{ssec:non-smooth_method} with $\tols{} = 0.5$, and a variant of latter without the normalization given by~\cref{eq:non-smooth_normalization}.
  }
  \label{fig:non-smooth_results}
\end{figure}
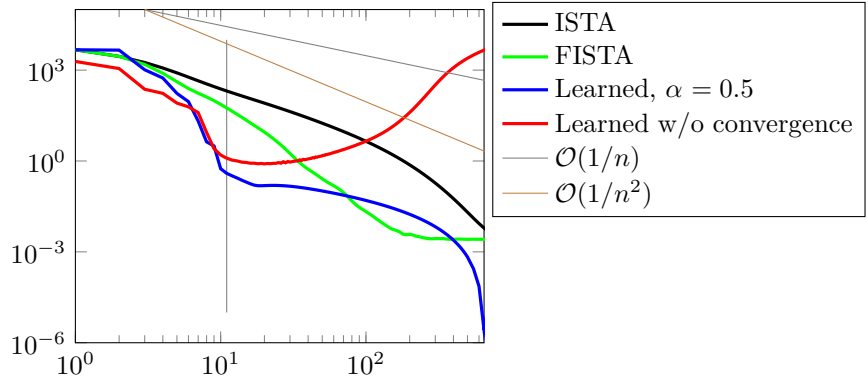

\Cref{fig:non-smooth_results} visualizes the performance of different methods averaged over the set of test images. We identified the minimal objective value $F^*$ achieved among all the methods and plot the difference $F(x_n) - F^*$ to the minimal value in logarithmic scale depending on the iteration.

We can see that, during the first 10 iterations (marked with a grey vertical line), the learned methods perform significantly better than the two baseline methods: \ac{ISTA} \cite{Goldstein:1964} and \ac{FISTA} \cite{BeckTeboulle:2009}.
Moreover, the normalization step \cref{eq:non-smooth_normalization} not only ensures convergence, it also improves the performance after 10 iterations.

\begin{figure}
  \centering
  \begin{tikzpicture}
    \begin{loglogaxis}
      [
      ymin=0.000001,
      ymax=100000,
      xmin=1,
      xmax=650,
      height=6cm,
      legend style = {anchor=west, at={(1.02, 0.70)}},
      legend cell align={left},
      no markers,
      cycle list={
        {black, solid, very thick},
        {green, solid, very thick},
        {blue, solid, very thick},
        {gray, solid},
        {brown, solid}}
      ]
      \draw[gray] (11, .00001) -- (11, 10000);
      \addplot table {plot_data/nonsmooth_case/ista.tab};
      \addplot table  {plot_data/nonsmooth_case/learned_dx2_e.001.tab};
      \addplot table {plot_data/nonsmooth_case/learned_dx2_a.5.tab};
      \addplot (3, 100000) -- (30, 10000) -- (300, 1000) -- (3000, 100);
      \addplot (3, 100000) -- (30, 1000) -- (300, 10) -- (3000, .1);
      \legend{ISTA\\ Learned,  $\tols{} = 0.999$\\ Learned,  $\tols{} = 0.5$\\  $\mathcal{O}(1/n)$\\ $\mathcal{O}(1/n^2)$\\}
    \end{loglogaxis}
  \end{tikzpicture}

  \caption{Plot of $F\p{x_n} - F^*$ as a function of the iteration number for \ac{ISTA} \cite{Goldstein:1964} and the learned scheme in \cref{ssec:non-smooth_method} with convergent updates by  \cref{eq:non-smooth_normalization} for different values of $\tols{}$.}
  \label{fig:non-smooth_ablation}
\end{figure}
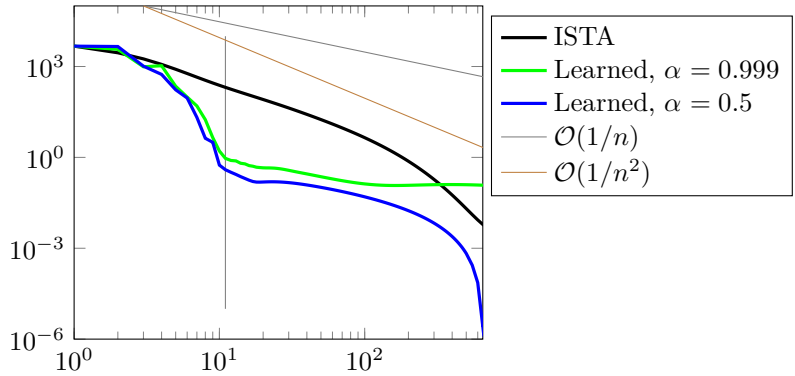

Next, we show the effect of choosing different normalization parameters $\tols{}$ in~\cref{fig:non-smooth_ablation}. As in smooth optimization, decreasing $\tols{}$ improves convergence. Moreover, in this case $\tols{} = 0.5$ leads to the best performance among the tested methods during the first 10 iterations as well. 

\section{Conclusion}

Our goal was to construct deep-learning accelerated algorithms for convex optimization, and we began by noting that these heavily over-parametrized algorithms do not fit into classical optimization schemes, which only use a few parameters.

To do this, we proved convergence of two convex optimization algorithms, one for smooth optimization and one for non-smooth. Instead of showing their convergence under some specific parameter choice, we showed convergence whenever the updates are chosen in particular sets of full dimension. Given that these sets are rather large, there is hence an enormous number of options to choose from. To make the best choice in any situation we used a neural network and trained it to pick the updates for optimal convergence speed.

We applied our algorithms to the problem of regularized image reconstruction in \ac{CT} and showed that our algorithm out-performs classical algorithms, especially when only a few iterations are applied. Furthermore, our algorithm converged to the exact solution, as expected.

We hope that this work will open new venues in combining deep learning with provable convergence, and we see several interesting open problems arising from it. For one, our algorithms are only applicable to a subset of all convex problems and it would be particularly interesting to find a primal-dual scheme which can be combined with deep learning. Another interesting research direction is to investigate other definitions of \emph{the best} optimization algorithm, for example one highly sought-after feature is the ability to train our network to converge to a given accuracy as quickly as possible. Finally, one could combine the deterministic acceleration in the spirit of Nesterov~\cite{Nesterov:1983} or \cite{BeckTeboulle:2009} with the deep learning techniques introduced in this paper and find an optimal trade-off between the two.

\section{Acknowledgments}
Sebastian Banert was supported by the Swedish Foundation of Strategic Research grant AM13-0049 and by the \ac{WASP}.
Jevgenija Rudzusika was supported by the Swedish Foundation of
Strategic Research grant AM13-0049, grant from the VINNOVA Open Innovation Hub project 2015-06759, and by Philips Healthcare.
Jonas Adler was supported by the Swedish Foundation of
Strategic Research grant AM13-0049, Industrial PhD grant
ID14-0055, and by Elekta Instrument AB.
Finally, Ozan Öktem was supported by the Swedish Foundation of
Strategic Research grant AM13-0049.

We would like to thank Pooria Joulani and Pushmeet Kohli for reviewing and helping us improve the manuscript.

\printbibliography
\end{document}